\documentclass{article}

\usepackage{graphicx}
\usepackage{color}
\usepackage{indentfirst}
\usepackage{amsmath,amsfonts,amsthm,amssymb}
\usepackage{mathrsfs}
\usepackage{amscd}
\usepackage{rotating}
\usepackage{hyperref}

\def\co{\colon\thinspace}
\DeclareMathAlphabet{\mathsfsl}{OT1}{cmss}{m}{sl}
\newcommand{\tensor}[1]{\mathsfsl{#1}}

\newtheorem{thm}{Theorem}[section]
\newtheorem{lem}[thm]{Lemma}
\newtheorem{cor}[thm]{Corollary}
\newtheorem{prop}[thm]{Proposition}

\newtheorem*{thm*}{Theorem}
\newtheorem{claim}[thm]{Claim}

\theoremstyle{definition}

\newtheorem{rem}[thm]{Remark}

\newtheorem{exam}[thm]{Example}

\begin{document}

\def\G{{\Gamma}}
  \def\d{{\delta}}
  \def\ci{{\circ}}
  \def\e{{\epsilon}}
  \def\l{{\lambda}}
  \def\L{{\Lambda}}
  \def\m{{\mu}}
  \def\n{{\nu}}
  \def\o{{\omega}}
  \def\O{{\Omega}}
  \def\Th{{\Theta}}\def\s{{\sigma}}
  \def\v{{\varphi}}
  \def\a{{\alpha}}
  \def\b{{\beta}}
  \def\p{{\partial}}
  \def\r{{\rho}}
  \def\ra{{\rightarrow}}
  \def\lra{{\longrightarrow}}
  \def\g{{\gamma}}
  \def\D{{\Delta}}
  \def\La{{\Leftarrow}}
  \def\Ra{{\Rightarrow}}
  \def\x{{\xi}}
  \def\c{{\mathbb C}}
  \def\z{{\mathbb Z}}
  \def\2{{\mathbb Z_2}}
  \def\q{{\mathbb Q}}
  \def\t{{\tau}}
  \def\u{{\upsilon}}
  \def\th{{\theta}}
  \def\la{{\leftarrow}}
  \def\lla{{\longleftarrow}}
  \def\da{{\downarrow}}
  \def\ua{{\uparrow}}
  \def\nwa{{\nwtarrow}}
  \def\swa{{\swarrow}}
  \def\nea{{\netarrow}}
  \def\sea{{\searrow}}
  \def\hla{{\hookleftarrow}}
  \def\hra{{\hookrightarrow}}
  \def\sl{{SL(2,\mathbb C)}}
  \def\ps{{PSL(2,\mathbb C)}}
  \def\qed{{\hspace{2mm}{\small $\diamondsuit$}\goodbreak}}
  \def\pf{{\noindent{\bf Proof.\hspace{2mm}}}}
  \def\ni{{\noindent}}
  \def\sm{{{\mbox{\tiny M}}}}
  \def\sch{{{\mbox{\tiny $\chi$}}}}
   \def\sf{{{\mbox{\tiny F}}}}
   \def\sc{{{\mbox{\tiny C}}}}
  \def\ke{{\mbox{ker}(H_1(\partial M;\2)\ra H_1(M;\2))}}
  \def\et{{\mbox{\hspace{1.5mm}}}}
 \def\sk{{{\mbox{\tiny K}}}}
 \def\sj{{{\mbox{\tiny J}}}}
\def\sp{{{\mbox{\tiny P}}}}
\def\sc{{{\mbox{\tiny C}}}}
\def\sy{{{\mbox{\tiny Y}}}}
\def\sa{{{\mbox{\tiny A}}}}

\title{Detection of  knots and a cabling formula for A-polynomials}

\author{{Yi Ni}\\{\normalsize Department of Mathematics, Caltech}\\
{\normalsize 1200 E California Blvd, Pasadena, CA 91125}
\\{\small\it Email\/:\quad\rm yini@caltech.edu}
\\\\
{Xingru Zhang}
\\
{\normalsize Department of Mathematics,
University at Buffalo}\\
{\small\it Email\/:\quad\rm xinzhang@buffalo.edu}}

\date{}
\maketitle

\begin{abstract}
We say that a given knot $J\subset S^3$ is detected by its knot Floer homology
and  $A$-polynomial if whenever a knot $K\subset S^3$  has the same knot Floer homology and the same $A$-polynomial
as $J$, then $K=J$. In this paper we show
that every torus knot $T(p,q)$ is detected by its knot Floer homology
and  $A$-polynomial. We also give a one-parameter family
of infinitely many hyperbolic knots in $S^3$ each of which is
detected by its knot Floer homology
and  $A$-polynomial.
In addition we give a cabling formula for the A-polynomials of
cabled knots in $S^3$, which is of independent interest. In particular we give explicitly  the A-polynomials of iterated torus knots.
\end{abstract}

\section{Introduction}
One of basic problems in knot theory is
to distinguish knots in $S^3$ from each other
using knot invariants.
There are several knot invariants
each being powerful enough to determine
if a given knot in $S^3$ is the unknot, such as the
knot Floer homology \cite{OSzGenus}, the A-polynomial \cite{BZ} \cite{DG}, and
the Khovanov homology \cite{KM}.
In other words, these invariants are each an unknot-detector.
It is also known that the knot Floer homology can detect
the trefoil knot and the figure 8 knot \cite{Gh}.
In this paper  we  first consider the
problem of detecting the set of torus knots $T(p,q)$ in $S^3$ using knot invariants.
To reach this goal  either the knot Floer homology or the $A$-polynomial  alone is not enough;
for instances the torus knot $T(4,3)$  has the same knot Floer homology as
the $(2,3)$-cable  over $T(3,2)$ \cite{Hed}, and the torus knot $T(15, 7)$
has the same $A$-polynomial as the torus knot $T(35, 3)$.
However when the two invariants are combined together,
the job can be done. We have

\begin{thm}\label{main1}If a knot $K$ in $S^3$  has the same  knot Floer homology and the same $A$-polynomial as a torus knot $T(p,q)$,
then $K=T(p,q)$.
\end{thm}

We then go further to find out a one-parameter family
of mutually distinct hyperbolic knots $k(l_*,-1,0,0)$ in $S^3$, where $l_*>1$ is integer valued,    each of which is  detected by
the combination of its $A$-polynomial and its knot Floer homology.
A knot diagram for $k(l_*,-1,0,0)$ is
illustrated in Figure~\ref{fig (l,-1)}.
Note that $k(2,-1,0, 0)$ is the $(-2,3,7)$-pretzel knot. Also note that
$k(l_*,-1,0,0)$ is a subfamily of the hyperbolic knots
$k(l,m,n,p)$ (with some forbidden values on integers $l,m,n,p$)  given in \cite{EM1}
each of which admits a half-integral toroidal surgery
(the  slope formula is given in \cite{EM2}, and recalled in
Section~\ref{topological properties of J} of this paper)
and by \cite{GL2} these hyperbolic knots $k(l,m,n,p)$ are the only hyperbolic knots
in $S^3$ which admit non-integral toroidal surgeries.

\begin{thm}\label{main2}The family of knots
$\{k(l_*,-1,0,0); l_*>1, l_*\in \z\}$ are mutually distinct hyperbolic knots in $S^3$.
 Let $J_*$ be any  fixed $k(l_*,-1,0,0)$, $l_*>1$.
 If a knot $K$ in $S^3$  has the same  knot Floer homology and the same $A$-polynomial as  $J_*$, then $K=J_*$.
\end{thm}

The key input from the knot Floer homology
in proving Theorems~\ref{main1} and \ref{main2} is that $K$ is fibred \cite{Gh,Ni} (since $T(p,q)$ or $J_*$ is) and has the same Alexander polynomial as
$T(p,q)$ or $J_*$ \cite{OS}.
With these conditions (conclusions)  in place, we can then
identify the knot using  the A-polynomial and the Alexander polynomial (in particular
 the genus of the knot).
The proof of Theorems~\ref{main1}  will be given in Section~\ref{proof1}
after we prepare some general properties on $A$-polynomials  in Section~\ref{Apoly}, where we
also derive a cabling formula  for $A$-polynomials of cabled knots in $S^3$ (Theorem~\ref{cabling formula})
and in particular we give explicitly  the A-polynomials of iterated torus knots
(Corollary~\ref{A-poly of iterated torus}).
The argument for Theorem~\ref{main2} is more involved than
that for Theorem~\ref{main1}, for which we need to  make
some more preparations (besides those
 made in Section~\ref{Apoly})  in the next
three sections.
In Section~\ref{topological properties of J}
we collect some topological properties about the family of knots $k(l,m,n,p)$,
in particular we give a complete genus formula
for $k(l,m,n,p)$ and show that  $k(l_*,-1,0,0)$ is a class of small
knots in $S^3$.
In Section~\ref{A-poly of J} we collect some
info about the $A$-polynomials of the knots $k(l,m,n,p)$
without knowing the explicit formulas of the $A$-polynomials,
and with such info we are able to show that if a hyperbolic knot $K$ has the same $A$-polynomial as a given $J_*=k(l_*,-1,0,0)$, then $K$ has the  same half-integral toroidal surgery slope
as $J_*$ and
$K$ is one of $k(l,m,0,p)$ with $l$ being divisible by $2p-1$.
We then in Section~\ref{distinguish J from k(l,m,0,p)}
identify each $J_*=k(l_*,-1,0,0)$ among  $k(l,m,0,p)$ with $(2p-1)|l$
using the genus formula and the half-integral toroidal slope formula
 for $k(l,m,0,p)$.
Results obtained in these three sections, together with some results from
 Section~\ref{Apoly}, are applied in Section~\ref{proof2}
to complete the proof of Theorem~\ref{main2}.

Note that the $A$-polynomial $A_\sk(x,y)$ (for a knot $K$ in $S^3$)
 used in this paper is a slightly modified version of the original $A$-polynomial given in \cite{CCGLS}.
The only difference is that in the current version, the $A$-polynomial of the
unknot is $1$ and  $y-1$ may possibly occur as a factor in $A_\sk(x,y)$
for certain knots $K$  contributed by some component of the character variety of the knot exterior
containing characters of irreducible
representations; while in the original version $y-1$ is a factor
of the $A$-polynomial for every knot contributed by the unique component
of the character variety of the knot exterior consisting of characters of reducible representations (see Section~\ref{Apoly} for details). The current version contains a bit more information than the original one.

\vspace{5pt}\noindent{\bf Acknowledgements.}\quad The first author was
partially supported by NSF grant
numbers DMS-1103976, DMS-1252992, and an Alfred P. Sloan Research Fellowship.


\section{Some properties of $A$-polynomials}\label{Apoly}

First we need to recall some background material on $A$-polynomials and set up some notations.
For a finitely generated group $\G$,  $R(\G)$ denotes the set of representations
(i.e. group homomorphisms) from  $\G$ to $SL_2(\c)$.
For each representation $\r\in R(\G)$, its character $\chi_\r$ is the complex valued
function
$\chi_\r:\G\ra \c$ defined by $\chi_\r(\g)=trace(\r(\g))$ for $\g\in \G$.
Let $X(\G)$ be the set of characters of representations in $R(\G)$
and $\displaystyle t:R(\G)\ra X(\G)$ the map sending $\r$ to $\chi_\r$.
Then both $R(\G)$ and $X(\G)$ are complex  affine
algebraic sets such that $t$ is a regular map (see \cite{CS} for details).

For an element $\g\in \G$, the function
$f_\g:X(\G)\ra \c$ is defined by $f_\g(\chi_\r)=(\chi_\r(\g))^2-4$ for each $\chi_\r\in X(\G)$.
Each  $f_\g$ is a regular function on $X(\G)$.
Obviously $\chi_\r\in X(\G)$ is a zero point of $f_\g$ if and only if either
$\r(\g)=\pm I$ or $\r(\g)$ is a parabolic element. It is also evident that $f_\g$ is invariant when $\g$ is replaced by  a conjugate of $\g$ or the inverse of $\g$.

Note that if
 $\displaystyle\phi:\G\ra \G'$ is a group homomorphism between two finitely generated groups,
 it naturally induces a regular map $\displaystyle\widetilde{\phi}: R(\G')\ra R(\G)$
 by $\displaystyle\widetilde{\phi}(\r')=\r'\circ\phi$
 and a regular map $\displaystyle\widehat{\phi}: X(\G')\ra X(\G)$ by
$\displaystyle\widehat{\phi}(\chi_{\r'})=\chi_{\widetilde{\phi}(\r')}$.
Note that if $X_0$ is an irreducible subvariety of $X(\G')$,
then the Zariski closure of $\displaystyle\widehat{\phi}(X_0)$ in $X(\G)$
is also irreducible.
 If in addition  the homomorphism $\phi$ is surjective, each of the regular maps
   $\displaystyle\widetilde{\phi}$ and $\displaystyle\widehat{\phi}$ is an embedding,
   in which case we may simply
consider $R(\G')$ and $X(\G')$ as subsets of $R(\G)$ and $X(\G)$ respectively,
and write $R(\G')\subset R(\G)$ and $X(\G')\subset X(\G)$.

For a compact manifold $W$, we use $R(W)$ and $X(W)$ to denote
$R(\pi_1(W))$ and $X(\pi_1(W))$ respectively.

The $A$-polynomial was introduced in \cite{CCGLS}. We slightly modify
its original definition for a knot $K$
in $S^3$ as follows. Let $M_\sk$ be the exterior of $K$ in $S^3$
and let $\{\m, \l\}$ be the standard meridian-longitude basis for $\pi_1(\p M_\sk)$.
Let $\widehat{i}_*: X(M_\sk)\ra X(\p M_\sk)$ be the regular map induced by the
inclusion induced homomorphism $i_*:\pi_1(\p M_\sk)\ra\pi_1(M_\sk)$, and let $\Lambda$ be the set of diagonal representations
of $\pi_1(\p M_\sk)$, i.e.
$$\displaystyle\Lambda=\{\r\in R(\p M_\sk); \;\; \mbox{$\r(\m)$, $\r(\l)$ are both diagonal matices}\}.$$
Then $\Lambda$ is a subvariety of $R(\p M_\sk)$ and
$\displaystyle t|_\Lambda:\Lambda\ra X(\p M_\sk)$ is a degree $2$, surjective, regular map.
We may identify $\Lambda$ with $\c^*\times \c^*$ through the eigenvalue  map
$\displaystyle E:\Lambda\ra \c^*\times \c^*$ which sends $\r\in \Lambda$ to
$\displaystyle (x,y)\in \c^*\times \c^*$ if $\displaystyle\r(\m)=\left(\begin{array}{cc}x&0\\0&x^{-1}
\end{array}\right)$ and $\displaystyle\r(\l)=\left(\begin{array}{cc}y&0\\0&y^{-1}
\end{array}\right)$.
For every knot in $S^3$, there is a unique component in $X(M_\sk)$
 consisting of  characters of reducible representations, which we call the trivial component of $X(M_\sk)$ (The trivial component is of dimensional one).
Now let $X^*(M_\sk)$ be the set of nontrivial components of $X(M_\sk)$ each of which  has
a $1$-dimensional image in
$X(\p M_\sk)$ under the map $\widehat{i}_*$.
The set $X^*(M_\sk)$ is possibly empty, and in fact with the current knowledge it is known
that $X^*(M_\sk)$ is empty if and only if $K$ is the unknot.
So when $K$ is a nontrivial knot,
$\displaystyle (t|_\Lambda)^{-1}(\widehat{i}_*(X^*(M_\sk))$ is $1$-dimensional in $\Lambda$
and in turn $\displaystyle E((t|_\Lambda)^{-1}(\widehat{i}_*(X^*(M_\sk)))$ is $1$-dimensional in
$\c^*\times \c^*\subset \c\times\c$.
Let $D$ be the Zariski closure of $\displaystyle E((t|_\Lambda)^{-1}(\widehat{i}_*(X^*(M_\sk)))$ in $\c^2$.
Then $D$ is a plane curve in $\c^2$ defined over $\q$.
Let $\displaystyle A_\sk(x,y)$ be the defining polynomial
of $D$ normalized so that  $A_\sk(x,y)\in \z[x,y]$ with no repeated factors and with $1$ as the
greatest common divisor of its coefficients.
Then $A_\sk(x,y)$ is uniquely defined  up to sign.
For the unknot we define its $A$-polynomial to be $1$.
As remarked in the introduction section,
 $y-1$ might occur as a factor of $A_\sk(x,y)$ for certain knots.
Also by \cite{BZ} \cite{DG}, $A_\sk(x,y)=1$ if and only if
$K$ is the unknot (in fact for every nontrivial knot $K$, $A_\sk(x,y)$
contains a nontrivial factor which is not $y-1$).

We note that from the constructional definition of the $A$-polynomial
we see that each component $X_0$ of $X^*(M_\sk)$ contributes a factor
$f_0(x,y)$ in $A_\sk(x,y)$, i.e.
$f_0(x,y)$ is the defining polynomial of the plane curve
$D_0$ which is the Zariski closure of $E((t|_\Lambda)^{-1}(\widehat{i}_*(X_0)))$ in $\c^2$,
and moreover $f_0(x,y)$ is balanced, i.e. if $(x,y)$ is a generic zero point of $f_0(x,y)$
 then $(x^{-1},y^{-1})$ is also a zero point of $f_0(x,y)$.
 Also note that $f_0(x,y)$ is not necessarily irreducible over $\c$
 but contains at most two irreducible factors over $\c$.
 We shall call such  $f_0(x,y)$ a
 balanced-irreducible factor of $A_\sk(x,y)$.
 Obviously $A_\sk(x,y)$ is a product of balanced-irreducible factors
 and the product decomposition is unique up to the ordering of the factors.

We now define  a couple of functions which will  be convenient
to use in  expressing the $A$-polynomials for torus knots and later on for
cabled knots and iterated torus knots.
Let $(p, q)$ be a pair of relative prime integers with $q\geq 2$. Define $F_{(p,q)}(x,y)\in \z[x,y]$ to be the polynomial
determined by the pair $(p,q)$ as follows:
\begin{equation}\label{polynomial of (p,q)}
F_{(p,q)}(x,y)= \left\{\begin{array}
 {ll}1+x^{2p}y,\;\;&  \mbox{if $q=2$, $p>0$,}\\
x^{-2p}+y,\;\;  &\mbox{if $q=2$,  $p<0$,}\\
-1+x^{2pq}y^2,\;\;&  \mbox{if $q>2$, $p>0$,}\\
-x^{-2pq}+y^2,\;\; & \mbox{if $q>2$, $p<0$}\end{array}\right.
\end{equation}
and define $G_{(p,q)}(x,y)\in \z[x,y]$ to be the polynomial
determined by the pair $(p,q)$ as follows:
\begin{equation}\label{G polynomial of (p,q)}
G_{(p,q)}(x,y)= \left\{\begin{array}{ll}
 -1+x^{pq}y,\;\;&  \mbox{if  $p>0$,}\\
-x^{-pq}+y,\;\; & \mbox{if  $p<0$.}\end{array}\right.
\end{equation}

Note that the ring $\c[x,y]$ is a unique factorization domain. The following lemma can be easily
checked.
\begin{lem}\label{lem:IrrFac}
Among the polynomials in (\ref{polynomial of (p,q)}) and (\ref{G polynomial of (p,q)}), the first two in (\ref{polynomial of (p,q)}) and the two in (\ref{G polynomial of (p,q)})
are irreducible over $\c$, and the last two in (\ref{polynomial of (p,q)}) can be factorized as the product of two irreducible polynomials over $\c$:
$$
\begin{array}{rcll}
-1+x^{2pq}y^2&=&(-1+x^{pq}y)(1+x^{pq}y),\;\;&  \mbox{if $q>2$, $p>0$,}\\
-x^{-2pq}+y^2&=&(-x^{-pq}+y)(x^{-pq}+y),\;\; & \mbox{if $q>2$, $p<0$.}\end{array}
$$
\end{lem}

The set of nontrivial torus knots $T(p,q)$
are naturally indexed by pairs $(p,q)$ satisfying $|p|>q\geq 2$, $(p,q)=1$.
Note that $T(-p,q)$ is the mirror image of $T(p,q)$.
The $A$-polynomial
of a torus knot $T(p,q)$  is given by
(e.g. \cite[Example 4.1]{Z}):
\begin{equation}\label{Apoly of T(p,q)}
A_{T(p,q)}(x,y)= F_{(p,q)}(x,y)
\end{equation}
In particular the $A$-polynomial distinguishes $T(p,q)$ from $T(-p, q)$.

For the exterior $M_\sk$ of a nontrivial knot in $S^3$,
 $H_1(\p M_\sk;\z)\cong\pi_1(\p M_\sk)$ can be considered as a subgroup
 of $\pi_1(M_\sk)$ which is well defined up to conjugation.
 In particular, the function $f_\a$ on $X(M_\sk)$ is well defined for each class $\a\in H_1(\p M_\sk;\z)$.
 As $f_\a$ is also invariant under the change of the orientation of $\a$,
 $f_\a$ is also well defined when $\a$ is a slope in $\p M_\sk$.
 Later on for convenience we  will often not make a distinction
 among  a primitive class of $H_1(\p M_\sk;\z)$, the corresponding
  element of $\pi_1(\p M_\sk)$ and the corresponding slope in $\p M_\sk$,
  so long as it is well defined.

 It is known (e.g. \cite{BZ4}) that any irreducible curve  $X_0$ in $X(M_\sk)$
  belongs to one of the following
 three mutually exclusive types:
\newline
 (a) for each slope $\a$ in $\p M_\sk$, the function $f_\a$ is non-constant on $X_0$;
\newline
 (b) there is a unique slope $\a_0$ in $\p M_\sk$ such that the function $f_{\a_0}$ is constant on $X_0$;
\newline
(c) for each slope $\a$ in $\p M_\sk$, the function  $f_\a$ is constant on $X_0$.
\newline
Obviously a curve   of
type (a) or (b) has one dimensional image in $X(\p M_\sk)$ under the map
$\widehat i_*$.
Note that the trivial component of $X(M_\sk)$ is of type (b).
Hence a curve of type (a) is contained in $X^*(M_\sk)$
and so is a curve of type (b) if it is not the trivial component
 of $X(M_\sk)$.

An irreducible curve  in case (a) is named a {\it norm curve}.
 Indeed as the name indicates,  a  norm curve in $X(M_\sk)$  can be used to define  a  norm, known as  Culler-Shalen  norm,  on the real $2$-dimensional
 plane $H_1(\p M_\sk; \mathbb R)$ satisfying certain properties.
 Such curve  exists when $M_\sk$ is hyperbolic, namely any component of $X(M_\sk)$
  which contains the character of a discrete faithful representation
 of $\pi_1(M_\sk)$ is a norm curve.

For an irreducible curve  $X_0$ in $X(M_\sk)$, let $\tilde X_0$ be the smooth projective
completion of $X_0$ and let $\phi:\tilde X_0\ra X_0$ be the birational
isomorphism.
The map $\phi$ is onto and is defined at all but finitely many points of $\tilde X_0$.
The points of $\tilde X_0$ where $\phi$  is not defined are called ideal points and all other points of $\tilde X_0$ are called regular points.
The  map $\phi$ induces an isomorphism
from  the function field of $X_0$ to that of $\tilde X_0$. In particular
every regular function $f_\g$ on $X_0$
corresponds uniquely to its extension $\tilde f_\g$ on $\tilde X_0$  which is a rational function.
 If $\tilde f_\g$ is not a constant function on $\tilde X_0$, its degree, denoted $\deg( \tilde f_\g)$,
 is equal to the number of zeros of $\tilde f_\g$ in $\tilde X_0$ counted with multiplicity, i.e.
 $$\deg(\tilde f_\g)=\sum_{v\in \tilde X_0} Z_v(\tilde f_\g),$$
where $Z_v(\tilde f_\g)$ is the zero degree of
$\tilde f_\g$ at point $v\in \tilde X_0$.

We shall identify $H_1(\p M_\sk; \mathbb R)$ with the real $xy$-plane
so that $H_1(\p M_\sk;\mathbb Z)$ are integer lattice points
with $\mu=(1,0)$ being the meridian class and $\l=(0,1)$ the longitude class.
So each slope $m/n$ corresponds to the pair of primitive elements
$\pm (m,n)\in H_1(\p M_\sk;\z)$.

\begin{thm}\label{properties of a norm curve}
 Let $X_0$ be a norm curve  of $X(M_\sk)$.
 Then the associated Culler-Shalen norm  $\|\cdot\|_0$ on $H_1(\p M_\sk;\mathbb R)$
 has the following properties: let $$s_0=\min\{\|\a\|_0; \a\ne 0, \a\in H_1(\p M_\sk; \z)\}$$
and let $B_0$ be the disk in $H_1(\p M_\sk;\mathbb R)$ with radius $s_0$ with
respect to the norm $\|\cdot\|_0$, then

(1) For each
nontrivial element $\a=(m,n)\in H_1(\p M_\sk;\z)$,
 $\|\a\|_0=\deg(\tilde f_\a)\ne 0$ and thus $\|\a\|_0=\|-\a\|_0$.
 \newline
(2) The disk $B_0$ is a convex finite sided polygon symmetric to the origin
whose interior does not contain any non-zero element of $H_1(\p M_\sk;\z)$
and whose boundary contains at least one but at most four
nonzero classes of $H_1(\p M_\sk;\z)$ up to sign.
\newline
(3)
If $(a,b)$ is a vertex of $B_0$, then there is a boundary slope $m/n$ of $\p M_\sk$
 such that $\pm (m,n)$ lie in the line passing through $(a,b)$ and $(0,0)$.
 (That is, $a/b$ is a boundary slope of $\p M_\sk$ for any vertex $(a,b)$ of $B_0$).
  \newline
(4)
If a primitive  class $\a=(m,n)\in H_1(\p M_\sk;\z)$ is not a boundary class
and $M_\sk(\a)$ has no noncyclic representations, then $\a=(m,n)$ lies in $\p B$
(i.e. $\|\a\|_0=s_0$)  and is not
a vertex of $B_0$.
\newline
(5) If the meridian class $\m=(1,0)$ is not a boundary class,
 then for any non-integral class $\a=(m,n)$ if it is not a vertex of $B_0$
 then it does not lie in $\p B$ and thus $\|\a\|_0>\|\m\|_0=s_0$.
\end{thm}

Theorem~\ref{properties of a norm curve} is originated
from
\cite[Chatper 1]{CGLS} although it was assumed there that the curve $X_0$
contains the character of a discrete faithful representation of $M_\sk$.
The version given here is contained in \cite{BZ4}.

Recall that  if $f_0(x,y)=\sum a_{i,j}x^iy^j\in \c[x,y]$ is a two variable polynomial
 in $x$ and $y$ with complex coefficients,
the  Newton polygon $N_0$ of $f_0(x,y)$ is defined to be the
convex hull in the real $xy$-plane of the set of points
$$\{(i,j); a_{i,j}\ne 0\}.$$

The following theorem  is proved in
\cite{BZ2}.
\begin{thm}\label{dual polygons}
Let $X_0$ be a norm curve   of $X(M_\sk)$ and let $f_0(x,y)$ be the balanced-irreducible factor of $A_\sk(x,y)$ contributed
by $X_0$.
Then the norm polygon $B_0$ determined by $X_0$ is dual to the Newton polygon $N_0$ of $f_0(x,y)$
in the following way: the set of
slopes of vertices  of
$B_0$ is equal to the set of slopes of edges of $N_0$.
 In fact $B_0$ and $N_0$ mutually determine each other up a positive integer
 multiple.
\end{thm}

We remark that although in \cite{BZ2} there was some additional condition
imposed on $X_0$ and the version of the $A$-polynomial defined in \cite{BZ2} is
mildly different from the  one as given here,
the above theorem remains valid with identical reasoning as given in \cite{BZ2}.
We only need to describe the exact relation
between $B_0$ and $N_0$ as follows, to see how they determine each other up to
an integer multiple.
Let $Y_0$ be the Zariski closure of  the restriction $\widehat{i}_*(X_0)$ of
$X_0$ in  $X(\p M_\sk)$ and let $d_0$ be the degree of the map $\widehat{i}_*:X_0\ra Y_0$.
As explained in \cite{BZ2} (originated from \cite{Shan}), the Newton polygon $N_0$ determines  a width
function $w$ on the set of slopes given by
$$w(p/q)=k\in \z$$
if $k+1$ is the number of lines in the $xy$-plane of the slope $q/p$
which contain points of both $\z^2$ and $N_0$. The width function
in turn defines a norm
$\|\cdot\|_{N_0}$ on the $xy$-plane $H_1(\p M_\sk;\mathbb R)$ such that
$$\|(p,q)\|_{N_0}=w(p/q)$$ for each primitive class $(p,q)\in H_1(\p M_\sk; \z)$.
Finally $$\|\cdot\|_0=2d_0\|\cdot\|_{N_0}.$$

\begin{cor}\label{non-hyperbolic}
If every balanced-irreducible factor of $A_\sk(x,y)$ over $\c$
has two monomials, then $K$ is not a hyperbolic knot.
\end{cor}

\pf The condition of the corollary
means that the Newton polygon of each balanced-irreducible factor of $A_\sk(x,y)$
consists of a single edge. On the other hand  for a hyperbolic knot,
its character variety contains a norm curve component $X_0$ which contributes
 a balanced-irreducible factor $f_0(x,y)$ to the $A$-polynomial such that
 the Newton polygon of $f_0(x,y)$ has at least two edges of different slopes.
\qed

An irreducible curve in $X^*(M_\sk)$ of type (b) (such curve  exists only for certain knots)
 is named a {\it semi-norm
curve} as suggested by the following theorem which is contained in \cite{BZ4}.
\begin{thm}\label{properties of a semi-norm curve}
Suppose that $X_0\subset X^*(M_\sk)$ is  an irreducible curve of type (b)  with   $\a_0$ being
 the unique slope such that $f_{\a_0}$ is constant on $X_0$. Then a semi-norm $\|\cdot\|_0$ can be defined on $H_1(\p M_\sk;\mathbb R)$, with the following properties:
\newline
(1) For each slope  $\a\ne \a_0$, $\|\a\|_0=\deg(\tilde f_\a)\ne 0$.
\newline
(2) For the unique slope $\a_0$ associated to $X_0$, $\|\a_0\|_0=0$ and $\a_0$ is a boundary slope  of $M_\sk$.
\newline
(3) If $\a$ is a primitive class and is not a boundary class
 and $M_\sk(\a)$ has no non-cyclic representation, then
 $\D(\a,\a_0)=1$.
 \newline
 (4) Let $s_0=min\{\|\a\|_0;\;\mbox{$\a\ne \a_0$ is a slope}\}$.
 Then for any slope $\a$,
 $\|\a\|_0=s_0\D(\a,\a_0)$.
\end{thm}

Note that for each torus knot $T(p,q)$, every nontrivial component
in its character variety is a semi-norm curve with $pq$ as the
associated slope.

\begin{rem}
If $K$ is a small knot, i.e. if its exterior $M_\sk$ does not contain
any closed essential surface,
then every nontrivial component of $X(M_\sk)$ is either a norm curve or a semi-norm
curve.
\end{rem}

We now proceed to get some  properties on $A$-polynomials of satellite knots in $S^3$.
Recall that a knot  $K$  in $S^3$ is a satellite knot if there is a pair of
knots $C$  and $P$ in $S^3$, called a  companion knot
and a pattern knot respectively,  associated to $K$, such that
$C$ is nontrivial, $P$  is contained  in a trivial solid torus $V$ in $S^3$
but is not contained in a
$3$-ball of $V$ and is not isotopic to the core circle of $V$,
and there is a homeomorphism $f$ from $V$ to a regular neighborhood $N(C)$ of $C$ in $S^3$
which maps a  longitude of $V$ (which bounds a disk in $S^3$) to a longitude of $N(C)$
(which bounds a Seifert surface for $C$) and maps a meridian of $V$ to a meridian  of $N(C)$,
 and finally  $K=f(P)$.
We sometimes  write a satellite knot as $K=(P, C, V, f)$ to include the above defining
information ($K$ still depends on how $P$ is embedded in $V$).

\begin{lem}\label{P-factor}Let $K=(P,C, V,f)$ be a satellite knot in $S^3$.
Then $A_\sp(x,y)|A_\sk(x,y)$ in $\z[x,y]$.\end{lem}

\pf The lemma is obviously true when $P$ is the unknot in $S^3$.
So we may assume that $P$ is a nontrivial knot in $S^3$.
Let $M_\sk$, $M_\sc$, $M_\sp$ be the exteriors of $K$, $C$ and $P$ in $S^3$ respectively.
  There is a degree one  map $h: (M_\sk, \p M_\sk)\ra (M_\sp, \p M_\sp)$ such that  $h^{-1}(\p M_\sp)=\p M_\sk$
  and $h|:\p M_\sk\ra \p M_\sp$ is a homeomorphism.
The  map $h$ is given by a  standard construction as follows.
From the definition of the pattern knot
given above, we see that if  $W$ is the exterior of $P$ in the trivial solid torus  $V$, then $M_\sp$ is obtained by Dehn filling $W$ along $\p V$ with a solid torus $V'$ such that the  meridian  slope of $V'$
is identified with the longitude slope  of $V$.
Also if we let $Y=f(W)$, then $M_\sk=M_\sc\cup Y$. Now the degree one map $h: M_\sk\ra M_\sp$ is defined to be: on $Y$ it is the homeomorphism $f^{-1}:Y\ra W$ and on $M_\sc$
it maps
 a regular neighborhood
  of a Seifert surface  in $M_\sc$  to a regular neighborhood of
  a meridian disk of $V'$ in $V'$ and maps the rest of $M_\sc$ onto
  the rest of $V'$ (which is a $3$-ball).

The degree one map $h$ induces a surjective homomorphism
$h_*$ from $\pi_1(M_\sk)$ to $\pi_1(M_\sp)$ such that
$$h_*|: \pi_1(\p M_\sk)\ra \pi_1(\p M_\sp)$$
 is an isomorphism, mapping the meridian to the meridian
 and the longitude to the longitude.
 In turn $h_*$ induces an embedding $\widehat{h_*}$ of $X(M_\sp)$ into $X(M_\sk)$
 in such a way that the restriction of $\widehat{h_*}(X(M_\sp))$ on $X(\p M_\sk)$
 with respect to the standard meridian-longitude basis $\{\m_\sk, \l_\sk)$ of $\p M_\sk$
 is the same as the restriction of $X(M_\sp)$ on $X(\p M_\sp)$ with respect to
 the standard meridian-longitude basis $\{\m_\sp, \l_\sp\}$ of $\p M_\sp$, i.e.
 for each $\chi_\r\in X(M_\sp)$, we have  $\chi_{\r}(\m_\sp)=\widehat{h_*}(\chi_\r)(\m_\sk)$,
 $\chi_{\r}(\l_\sp)=\widehat{h_*}(\chi_\r)(\l_\sk)$ and
 $\chi_{\r}(\m_\sp\l_\sp)=\widehat{h_*}(\chi_\r)(\m_\sk\l_\sk)$.
  The conclusion of the lemma now follows easily from the constructional definition
  of the $A$-polynomial.
\qed

For polynomials $f(x,\overline y)\in \c[x,\overline y]$ and $g(y, \overline y)\in\c[y,\overline y]$
both with nonzero degree in $\overline y$,
let $$Res_{\overline y}(f(x,\overline y), g(y, \overline y))$$
denote the resultant of $f(x,\overline y)$ and $g(y, \overline y)$
eliminating the variable $\overline y$.
In general  $Res_{\overline y}(f(x,\overline y), g(y, \overline y))$
may have repeated factors even when both $f(x,\overline y)$ and $g(y, \overline y)$
are irreducible over $\c$.
For a polynomial $f(x,y)\in \c[x,y]$, let $$Red[f(x,y)]$$ denote the polynomial obtained from
$f(x,y)$ by deleting all its repeated factors.

\begin{prop}\label{non-zero winding}Let $K=(P,C,V,f)$ be a satellite knot
 such that the winding number $w$  of $P$ in the solid torus $V$ is
non-zero.
Then every balanced-irreducible factor $f_\sc(\overline x,\overline y)$  of the $A$-polynomial $A_\sc(\overline x,\overline y)$ of $C$ extends to a balanced factor $f_\sk(x,y)$  of the $A$-polynomial  $A_\sk(x,y)$ of $K$.
More precisely,
\newline
(1) if the $\overline y$-degree of $f_\sc(\overline x,\overline y)$ is non-zero, then  $f_\sk(x,y)=Red[Res_{\overline y}(f_\sc(x^w,\overline y),\overline y^w-y)]$.
 In particular if $f_\sc(\overline x,\overline y)=\overline y+\d\overline x^n$ or $\overline y\;\overline x^n+\d$ for some nonnegative  integer $n$ and $\d\in\{1,-1\}$
 (such factor is irreducible and balanced), then $f_\sk(x,y)=y-(-\d)^wx^{n w^2}$
or $yx^{nw^2}-(-\d)^w$ respectively;
\newline
(2) if the $\overline y$-degree of $f_\sc(\overline x,\overline y)$ is zero, i.e.
$f_\sc(\overline x,\overline y)=f_\sc(\overline x)$ is a function of $\overline x$ only,
then  $f_\sk=f_\sc(x^w)$.
\end{prop}

\pf Let $M_\sk$, $M_\sc$, $Y$, $\m_\sc,\l_\sc$,
  $\m_\sk, \l_\sk$ be defined as in the proof of Lemma~\ref{P-factor}.
  We have  $M_\sk=M_\sc\cup Y$.
 Note that $H_1(Y;\z)=\z[\m_{\sk}]\oplus\z[\l_{\sc}]$, and
 $[\l_\sk]=w[\l_\sc]$, $[\m_\sc]=w[\m_\sk]$.
 Given  a balanced-irreducible factor $f_\sc(\overline x,\overline y)$ of $A_\sc(\overline x,\overline y)$, let $X_0$ be a component of $X^*(M_\sc)$ which gives rise
 the factor $f_\sc(\overline x,\overline y)$.
 For each element $\chi_\r\in X_0$, the restriction of $\r$ on $\pi_1(\p M_\sc)$ can be extended to an abelian representation of $\pi_1(Y)$, and thus $\r$ can be extended to a representation of $\pi_1(M_\sk)$,
  which we still denote by $\r$, such that
  $$\r(\m_\sc)=\r(\m_\sk^w),\;\r(\l_\sk)=\r(\l_\sc^w).$$
 It follows that
 $X_0$ extends to a component or components of $X^*(M_\sk)$ whose restriction on $\p M_\sk$ is or are  one dimensional and thus all together gives rise a balanced factor $f_\sk(x,y)$  of $A_\sk(x,y)$.
 Moreover   the variables $(x, y)$ of $f_\sk(x,y)$ and the
variables
 $\overline x$ and $\overline y$ of $f_\sc(\overline x,\overline y)$ are
 related by
 \begin{equation}\label{x and x bar}
 \overline x=x^w,\;y=\overline y^w.
 \end{equation}
Therefore when the $\overline y$-degree of $f_\sc(\overline x,\overline y)$ is positive,
 $f_\sk(x,y)$ can be obtained by
taking the resultant of $f_\sc(x^w,\overline y)$ and
 $\overline y^w-y$, eliminating the variable $\overline y$, and then
 deleting possible repeated factors.
  In particular if  $f_\sc(\overline x,\overline y)=\overline y+\d\overline x^n$ or $\overline y\;\overline x^n+\d$ for some nonnegative  integer $n$ and $\d\in\{-1,1\}$, then
  the resultant of $\overline y+\d x^{wn}$ or $\overline y  x^{wn}+\d$ with $\overline y^w-y$, eliminating the variable $\overline y$, is $y-(-\d)^wx^{n w^2}$
or $yx^{nw^2}-(-\d)^w$  respectively (which is irreducible over $\c$ and is banlanced). Also if the degree of $f_\sc$ in $\overline y$ is zero,
then obviously $f_\sk=f_\sc(x^w)$.
\qed

\begin{rem}{\rm
In Proposition~\ref{non-zero winding}, the $y$-degree of $f_\sk(x,y)$ is at most
equal to the $\overline y$-degree of $f_\sc(\overline x, \overline y)$ (and generically
they are equal).
This follows directly from the definition of the resultant (cf. \cite[IV,\S8]{Lang}).}
\end{rem}

Next we are going to consider
cabled knots.
Let $(p,q)$ be a pair of relatively prime  integers with $|q|\geq 2$,
 and $K$ be the $(p,q)$-cabled knot  over a nontrivial knot $C$.
That is, $K$ is a satellite knot with $C$ as a companion knot
and with $T(p,q)$ as a pattern knot which lies in the defining
solid torus $V$ as a standard $(p,q)$-cable with  winding number $|q|$.
As the $(-p,-q)$-cable over a knot  is equal to the $(p,q)$-cable over
the same knot, we may always assume $q\geq 2$.
The following theorem gives a cabling formula for the $A$-polynomial
of a cabled knot $K$ over a nontrivial knot $C$, in terms of the $A$-polynomial
$A_\sc(\overline x, \overline y)$ of $C$.

\begin{thm}\label{cabling formula}
Let $K$ be the $(p,q)$-cabled knot over a nontrivial knot $C$, with $q\geq 2$.
 Then $$A_\sk(x,y)=Red[F_{(p,q)}(x,y)Res_{\overline y}(A_\sc(x^q,\overline y), \overline{y}^q-y)]$$
if  the $\overline y$-degree of $A_\sc(\overline x, \overline y)$ is nonzero
and $$A_\sk(x,y)=F_{(p,q)}(x,y)A_\sc(x^q)$$
 if  the $\overline y$-degree of $A_\sc(\overline x, \overline y)$ is zero.  \end{thm}

\pf
For a polynomial $f(\overline x, \overline y)\in \c[\overline x, \overline y]$,
define $$Ext^q[f(\overline x, \overline y)]
=\left\{\begin{array}{ll}
Red[Res_{\overline y}(f(x^q,\overline y),  \overline{y}^q-y)],
&\mbox{if the degree of $f(\overline x, \overline y)$ in $\overline y$
is nonzero},\\
f(x^q), &\mbox{if the degree of $f(\overline x, \overline y)$ in $\overline y$
is zero}
\end{array}\right.$$
Then Proposition~\ref{non-zero winding} was saying that
$$f_\sk(x,y)=Ext^w[f_\sc(\overline x, \overline y)]$$
and Theorem~\ref{cabling formula} is saying that
$$A_\sk(x,y)=Red[F_{(p,q)}(x,y) Ext^q[A_\sc(\overline x, \overline y)]].$$
Note that
$Ext^q[A_\sc(\overline x,\overline y)]=Red[\prod Ext^q[f_\sc(\overline x,\overline y)]]$
where the product runs over all balanced-irreducible factors $f_\sc(\overline x, \overline y)$
of $A_\sc(\overline x, \overline y)$ and thus by Proposition~\ref{non-zero winding}, $Ext^q[A_\sc(\overline x, \overline y)]$ is a balanced factor
of $A_\sk(x,y)$.
 So we only need to show

\begin{claim}\label{F is a factor of A_K}
 $F_{(p,q)}(x,y)$ is  a balanced factor of
$A_\sk(x,y)$. (Note that each irreducible factor of $F_{(p,q)}(x,y)$
is balanced).
\end{claim}

\begin{claim}\label{no other factor}
 Besides $F_{(p,q)}(x,y)$ and
$Ext^q[A_\sc(\overline x, \overline y)]$,
$A_\sk(x,y)$ has no other balanced factors.
\end{claim}

To prove the above two claims, let $M_\sk$, $M_\sc$, $Y$, $\m_\sc,\l_\sc$,
  $\m_\sk, \l_\sk$ be defined as in Lemma~\ref{P-factor} with respect to
   $K=(P, C, f, V)$ where $P=T(p,q)$ is embedded in $V$ as
  a standard  $(p,q)$-cable.
We have $M_\sk=M_\sc\cup Y$ and  $\p Y=\p M_\sc\cup \p M_\sk$.
For a convenience to the present argument, we give a direct description of $Y$
as follows. We may consider  $N=N(C)$ as $D\times C$, where $D$
is a disk of radius $2$, such that $\{x\}\times C$ has slope zero for
each point $x\in \p D$. Let
$D_*$ be the concentric sub-disk in $D$  with radius $1$.
Then $N_*=D_*\times C$ is a solid torus in $N(C)$ sharing the same core circle $C$.
We may assume that the knot $K$ is embedded in the boundary of $N_*$ as a standard $(p,q)$-curve, where $\p N_*$ has the meridian-longitude
coordinates consistent with that of $\p M_\sc=\p N$ (i.e. for a point $x\in \p D_*$, $\{x\}\times C$
is a longitude of $\p N_*$).
Then $Y$ is the exterior of $K$ in $N$.

Note that $Y$ is a Seifert fibred space whose base orbifold is an annulus with a single cone point of order $q$,
 a Seifert fiber of $Y$ in $\p M_\sc$ has slope $p/q$
 and  a Seifert fiber of $Y$ in $\p M_\sk$ has slope $pq$.
Let $\g_\sc$ be a Seifert fibre of $Y$ lying in $\p M\sc$
and $\g_\sk$ be a Seifert fibre of $Y$ lying in $\p M\sk$.
Up to conjugation, we may consider $\g_\sc$ and $\g_\sk$ as elements of $\pi_1(Y)$. Also note that $\g_\sc$ is conjugate to
 $\g_\sk$ in $\pi_1(Y)$.
It is well  known that each of $\g_\sc$ and $\g_\sk$ lies in the center of $\pi_1(Y)$ which is independent of  conjugation.
It follows that if $\r\in R(Y)$ is an  irreducible representation
then $\r(\g_\sc)=\r(\g_\sk)=\e I$,  for some fixed
$\e\in\{1,-1\}$, where  $I$ is the identity matrix. Hence if $X_0$ is an
irreducible subvariety of $X(Y)$ which contains
the character of an irreducible representation,
then for every $\chi_\r\in X_0$, we have $\r(\g_\sc)=\r(\g_\sk)=\e I$, which is
due to the fact that the characters of irreducible representations are
dense in $X_0$.

Now we are ready to prove Claim~\ref{F is a factor of A_K}.
If $|p|>1$, then by Lemma~\ref{P-factor} and the formula (\ref{Apoly of T(p,q)}), $F_{(p,q)}(x,y)$ is a
factor of $A_\sk(x,y)$. So   we may assume that $|p|=1$.
Under this assumption, one can see that the fundamental group of $Y$ has the following
presentation:
\begin{equation}\label{presentation of for fund gp of Y}\pi_1(Y)=<\a, \b\; |\; \a^q\b=\b \a^q>
\end{equation}
such that  $$\m_\sk=\a\b$$
where $\a$ is a based simple loop  free homotopic in $Y$ to the center circle of $N$
and $\b$ is a based simple loop free homotopic  in $Y$ to $\l_\sc$  in
$\p M_\sc$.
To see these assertions, note that $Y$ contains the essential annulus $A_2=\p N_*\cap Y$
with $\p A_2\subset \p M_\sk$ of the slope $pq$ (the cabling annulus, consisting of
Seifert fibers of $Y$) and $A_2$ decompose $Y$ into two pieces $U_1$ and  $U_2$,
such that $U_1$  is a
solid torus (which is $N_*\cap Y$) and $U_2$  is topologically $\p M_\sc$ times an
interval. The above presentation for $\pi_1(Y)$ is obtained by
applying the Van Kampen  theorem associated to the splitting of $Y=U_1\cup_{A_2} U_2$
along $A_2$.
We should note that as $|p|=1$, a longitude in $\p N_*$
intersects $K$ geometrically exactly once. It is this curve pushed into $U_1$
which yields the element $\a$ and pushed into $U_2$
which yields $\b$.
Also because $|p|=1$, $\{\g_\sc, \b\}$ form a basis for $\pi_1(U_2)$.
Therefore by Van Kampen, $\pi_1(Y)$
is generated by $\a, \b, \g_\sc$ with relations $\a^q=\g_\sc$ and $\g_\sc\b=\b\g_\sc$, which yields presentation (\ref{presentation of for fund gp of Y}) after canceling
the element $\g_\sc$.  With a suitable choice of orientation for $\b$
i.e. replacing $\b$ by its inverse if necessary, we also have $\m_\sk=\a\b$.
See Figure~\ref{cross-section} for an illustration when $q=3$.

\begin{figure}[!ht]
\centerline{\includegraphics{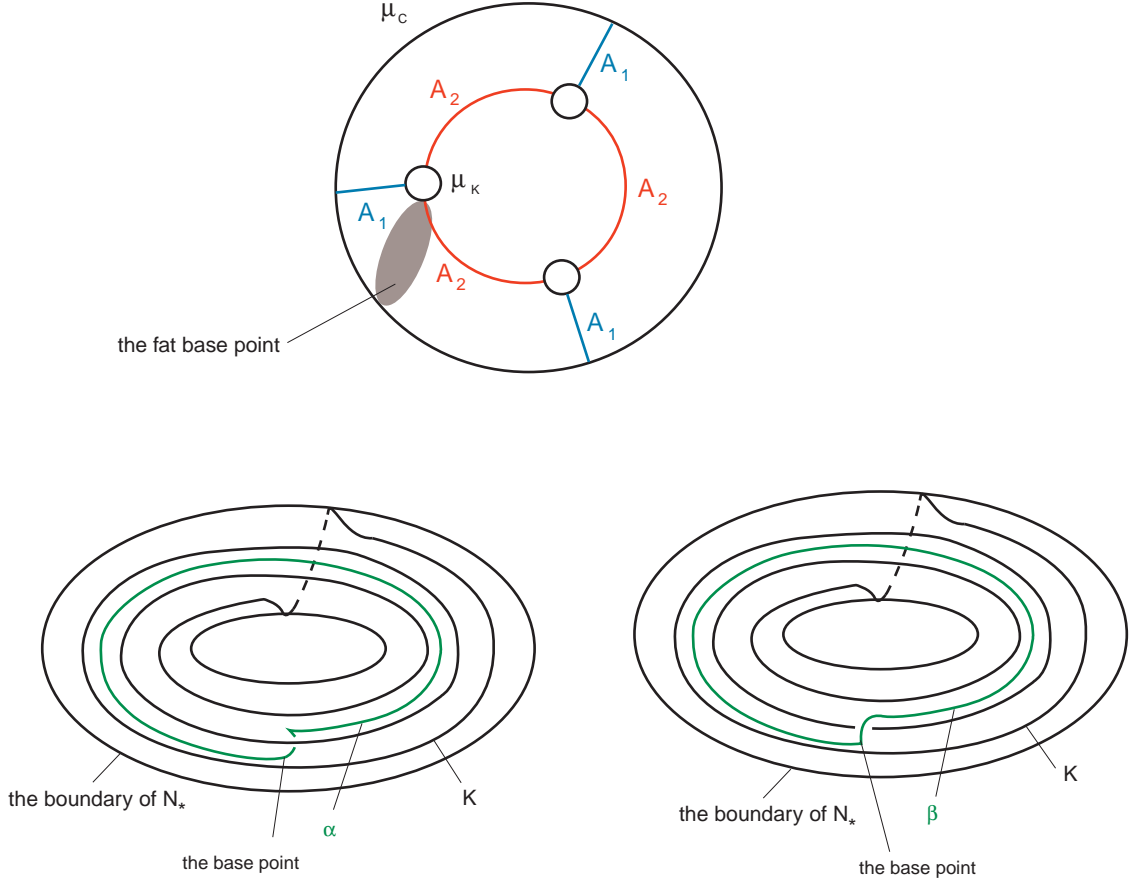}} \caption{Illustration of the
 cross section $D\cap Y$ of $Y$  and the elements $\a$ and  $\b$ of
$\pi_1(Y)$ when $q=3$.}\label{cross-section}
\end{figure}

 By \cite[Lemma 7.2]{Go}, $N_\sk(pq)$, which denotes the manifold obtained by Dehn surgery on $K$ in the solid torus $N$ with the slope $pq$, is homeomorphic to $L(q,p)\# (D^2\times S^1)$,
 and $M_\sk(pq)$, which denotes the manifold obtained by  Dehn filling of $M_\sk$ with the slope $pq$,
  is homeomorphic to $L(q,p)\# M_\sc(p/q)$ (where the meaning of the notation  $M_\sc(p/q)$ should be obvious).
As $|p|=1$, $M_\sc(p/q)$ is a homology sphere.
By \cite{KM}, $R(M_\sc(p/q))\subset R(M_\sc)$ contains at least one
irreducible representation $\r_\sc$.
Note that the restriction of $\r_\sc$ on  $\pi_1(\p M_\sc)$
is not contained in $\{I, -I\}$.
That is, we have:  $\r_\sc(\g_\sc)=I$ and $\r_\sc(\b)\ne \pm I$
(we may consider $\b$ as the longitude  $\l_\sc$  of $\pi_1(\p M_\sc)$).
If $q>2$, we can extend $\r_\sc$ to a curve of representations
of $\pi_1(M_\sk)$ with one dimensional  characters. In fact
for every $A\in SL_2(\c)$, we may define $\r_{\sa}\in R(M_\sk)$ as follows.
On $\pi_1(M_\sc)$, let $\r_\sa=\r_\sc$,
so in particular $\r_\sa(\b)=\r_\sc(\b)$,  define $\r_\sa(\a)=ABA^{-1}$
where $B$ is a fixed order $q$ matrix in $SL_2(\c)$.
It's routine to check that $\r_\sa$ is well defined,
 the trace of $\r_\sa(\m_\sk)=ABA^{-1}\r_\sc(\b)$
varies as $A$ runs over $SL_2(\c)$,
and $\r_\sa(\g_\sk)=I$.
Hence we get a  curve in $X(M_\sk)$ whose restriction on $\p M_\sk$
is one-dimensional, and moreover this curve generates
 $-1+yx^q$ as a factor of $A_\sk(x,y)$ if $p=1$
or generates  $-x^{q}+y$ as a factor of $A_\sk(x,y)$ if $p=-1$.

We now show that $1+yx^q$ is a factor of $A_\sk(x,y)$ if $p=1$
or $x^{q}+y$ is a factor of $A_\sk(x,y)$ if $p=-1$, for all $q\geq 2$.
For the representation $\r_\sc$ given in the
last paragraph, defined $\r_\sc^\e\in R(M_\sc)$  by
$$\r_\sc^\e(\g)=\e(\g)\r_\sc(\g),\; \mbox{for $\g\in \pi_1(M_\sc)$}$$
where $\e$ is the onto  homomorphism
$\e:\pi_1(M_\sc)\ra \{I,-I\}$.
As $\m_\sc$ is a generator of $H_1(M_\sc; \z)$,
$\r_\sc^\e(\m_\sc)=-\r_\sc(\m_\sc)$. Similarly as $\l_\sc$ is trivial in
$H_1(M_\sc;\z)$, $\r_\sc^\e(\l_\sc)=\r_\sc(\l_\sc)$.
Hence $\r_\sc^\e(\g_\sc)=\r_\sc^\e(\l_\sc^q\m_\sc^p)=-
\r_\sc(\g_\sc)=-I$.
We can now extend $\r_\sc^\e$ to $\r^\e_\sa$ over $M_\sk$ similarly as
for $\r_\sc$ to $\r_\sa$, only this time we choose $B$ as a fixed order $2q$ matrix in $SL_2(\c)$, so that
the trace of $\r^\e_\sa(\m_\sk)=ABA^{-1}\r_\sc(\b)$
varies as $A$ runs over $SL_2(\c)$,
and $\r^\e_\sa(\g_\sk)=-I$. The existence of the
factor of $A_\sk(x,y)$ that we set to prove now follows.
This completes the proof of Claim~\ref{F is a factor of A_K}.

We lastly  prove Claim~\ref{no other factor}. Given a balanced-irreducible factor $f_0(x,y)$ of $A_\sk(x,y)$, let $X_0$ be an irreducible component of $X^{*}(M_\sk)$ over $\c$ which produces $f_0(x,y)$.
Let $X_0^\sy$ and $X_0^\sc$ be the Zariski closure of the restriction of $X_0$ on $Y$ and $M_\sc$ respectively.
Note that each of $X_0^\sy$ and $X_0^\sc$ is irreducible.
Also  $X_0^\sy$ is at least one dimensional since its restriction  on $\p M_\sk$ is
one dimensional.
If $X_0^\sy$ does not contain
irreducible characters, then the restriction of $X_0^\sy$ on
$\p M_\sc$ is also one dimensional. So the restriction of $X_0^\sc$ on $\p M_\sc$ is one dimensional and $X_0$ is an abelian extension of $X_0^\sc$ (since the character of a reducible representation is also the character of an abelian representation).
It  follows that $X_0^\sc$ cannot be the trivial component in $X(M_\sc)$
for otherwise  $X_0$ woulde be the trivial component of $X(M_\sk)$.
Hence $X_0^\sc\in X^*(M_\sc)$ and $X_0$ is its abelian extension, which means that  $f_0(x,y)$ is a factor of $Ext^q[A_\sc(\overline x,\overline y)]$.

Hence we may assume that $X^\sy_0$ contains irreducible characters. Then, as noted before,
we have
$\r(\g_\sk)=\e I$ for each $\r$ with $\chi_\r\in X_0^\sy$.
If $\e=-1$, then $1+x^{pq}y$ or $x^{pq}+y$ is the factor contributed by $X_0$ to
 $A_\sk(x,y)$ corresponding to $p$ is positive or negative respectively.
 That is, $f_0(x,y)$ is a factor of $F_{(p,q)}(x,y)$.
 Hence we may assume that $\e=1$. It follows that   $X_0^\sy$ is a positive dimensional component of
 $N_\sk(pq)$.
 Therefore $q>2$ (because $N_\sk(pq)=L(q,p)\#(D^2\times S^1)$) and $X_0$ contributes the factor $-1+x^{pq}y$ or $-x^{pq}+y$  to
 $A_\sk(x,y)$ corresponding to $p$ is positive or negative respectively.
 That is, $f_0$ is a factor of $F_{(p,q)}(x,y)$.  This proves Claim~\ref{no other factor}
 and also completes the proof of the theorem.

Perhaps we should  note that in the above proof consistent choice of  base points can be made  for   all the relevant manifolds
 such as $M_\sk$, $M_\sc$, $Y$, $A_2$, $U_1$, $U_2$, $\p M_\sc$, $\p M_\sk$,
 $N_\sk(pq)$, $M\sc(p/q)$, $M_\sk(pq)$ so that
their  fundamental groups are all well defined as relevant subgroups or quotient groups.
 In fact we can choose a simply connected region in $D\cap Y$ (as shown in Figure~\ref{cross-section}) so that its intersection with each relevant manifold listed above
 is a simply connected region which is served as the `fat base point' of that manifold.\qed

\begin{exam}
{\rm Let $C$ be the figure $8$ knot. Its $A$-polynomial is
 $A_\sc(x,y)=x^{4} +(-1+  x^2 + 2 x^4 + x^{6}  - x^{8}) y + x^{4} y^2$ (\cite[Appendix]{CCGLS}).
  If  $K$ is the $(p,2)$-cable over $C$, $p>0$, then
$$\begin{array}{ll}
A_\sk(x,y)&=Red[F_{(p,2)}(x,y)Res_{\overline y}(A_\sc(x^2,\overline y), \overline{y}^2-y)]\\
&=(1+x^{2p}y)[x^{16} +(-1+  2 x^4 + 3 x^8 - 2 x^{12}  - 6 x^{16} - 2 x^{20} +
 3 x^{24}  + 2 x^{28} - x^{32}) y + x^{16} y^2].\end{array}$$
 If  $K$ is the $(p,3)$-cable over $C$, $p>0$, then
$$\begin{array}{ll}A_\sk(x,y)&=Red[F_{(p,3)}(x,y)Res_{\overline y}(A_\sc(x^3,\overline y), \overline{y}^3-y)]\\
&=(-1+x^{6p}y^2)[x^{36} +(-1+3 x^6 + 3 x^{12}- 8 x^{18} - 12 x^{24} + 6 x^{30}+
 20 x^{36} + 6 x^{42} - 12 x^{48} \\&\;\;\; - 8 x^{54}+ 3 x^{60}+ 3 x^{66}  - x^{72}) y+ x^{36} y^2].\end{array}$$}
\end{exam}

 Finally we would like to give an explicit formula
 for  the  $A$-polynomials of iterated torus knots.
Let $$K=[(p_1, q_2), (p_{2}, q_{2}),...,(p_n, q_n)]$$
 be an $n$-th iterated torus knot, i.e.
 $T(p_n, q_n)$ is a nontrivial torus knot and when $n>1$,
for each $i$, $n> i\geq 1$,
$[(p_i, q_i), (p_{i+1}, q_{i+1}),...,(p_n, q_n)]$
is  a satellite knot with $[(p_{i+1}, q_{i+1}),...,(p_n,q_n)]$ as a companion knot
and with $T(p_i,q_i)$, $q_i>1$,  as a pattern knot lying in the trivial solid torus $V$
as a $(p_i, q_i)$-cable with winding number $q_i$, where  $|p_i|$ may be less than $q_i$
and $|p_i|=1$ is also allowed.

\begin{cor}\label{A-poly of iterated torus}Let $K=[(p_1, q_1), (p_{2}, q_{2}),...,(p_n,q_n)]$ be an iterated torus knot.
If for each $1\leq i<n$, $q_i$ is odd, then $$A_{K}(x,y)=F_{(p_1,q_1)}(x,y)F_{(p_{2},q_{2})}(x^{q_1^2}, y)F_{(p_{3},q_{3})}(x^{q_1^2q_2^2}, y)\cdots F_{(p_n,q_n)}(x^{q_1^2q_2^2\cdots q_{n-1}^2}, y),$$
and if for some $1\leq i<n$, $q_i$ is even, we let $m$ be the smallest such integer,
then
$$\begin{array}
{ll}A_{K}(x,y)=&F_{(p_1,q_1)}(x,y)F_{(p_{2},q_{2})}(x^{q_1^2}, y)F_{(p_{3},q_{3})}(x^{q_1^2q_2^2}, y)\cdots
F_{(p_m,q_m)}(x^{q_1^2q_2^2\cdots q_{m-1}^2}, y)\\&
G_{(p_{m+1}, q_{m+1})}(x^{q_1^2q_2^2\cdots q_{m}^2}, y)
G_{(p_{m+2}, q_{m+2})}(x^{q_1^2q_2^2\cdots q_{m+1}^2}, y)
\cdots G_{(p_{n}, q_{n})}(x^{q_1^2q_2^2\cdots q_{n-1}^2}, y)
\end{array}
  $$\end{cor}

\begin{rem}
{\rm (1) In the $A$-polynomial given in the corollary,
each $$F_{(p_i, q_i)}(x^{q_1^2q_2^2\cdots q_{i-1}^2},y)\;\;
\mbox{or}\;\;G_{(p_i, q_i)}(x^{q_1^2q_2^2\cdots q_{i-1}^2},y)$$
is a nontrivial polynomial  even when $|p_i|=1$.

(2) The polynomial expression for $A_K(x, y)$ given in the corollary
has no repeated factors.

(3) The boundary slopes detected by $A_K(x,y)$ are precisely
the following $n$ integer slopes: $$p_1q_1,\; p_{2}q_{2}q_1^2,\; p_{3}q_{3}q_{1}^2q_2^2,\;...,\;
p_nq_nq_1^2q_2^2\cdots q_{n-1}^2.$$}\end{rem}

\pf The proof goes by induction on $n$ applying Theorem~\ref{cabling formula}.
When $n=1$, the proposition holds obviously.
Suppose for $n-1\geq 1$ the proposition holds.
Note that $K$ has  $C=[(p_{2},  q_{2}),...,(p_n,q_n)]$ as a companion knot and
$P=T(p_1, q_1)$  as the corresponding  pattern knot (which maybe a trivial knot in $S^3$, which occurs exactly when $|p_1|=1$).
By induction, the $A$-polynomial $A_\sc(x,y)$ of $C$
 is of the  corresponding  form as described by the corollary.
Now applying Theorem~\ref{cabling formula} one more time  to the pair $(C, P)$
 we see that the corollary holds.
We omit the routine details.
\qed

For instance the A-polynomial of the  $(r,s)$-cable over the $(p,q)$-torus
knot is
$$A(x,y)=\left\{\begin{array}{ll}F_{(r,s)}(x,y)F_{(p,q)}(x^{s^2},y), &
\mbox{if $s$ is odd},\\
F_{(r,s)}(x,y)G_{(p,q)}(x^{s^2},y), &
\mbox{if $s$ is even}.\end{array}\right.$$


\section{Proof of Theorem ~\ref{main1}}\label{proof1}

Suppose that $K$ is a knot in $S^3$ with the same knot Floer homology and
the same  A-polynomial as a given torus knot $T(p,q)$.
Our goal is to show that $K=T(p,q)$.

By (\ref{Apoly of T(p,q)}) and Lemma~\ref{non-hyperbolic}, $K$ is not a hyperbolic knot.
So  $K$  is either a torus knot or a satellite knot.

\begin{lem}\label{torus knot case}Suppose that $T(r,s)$ is a torus knot whose $A$-polynomial divides that of
$T(p,q)$ and whose Alexander polynomial divides that of $T(p,q)$, then
$T(r,s)=T(p,q)$.
\end{lem}

\pf Since $A_{T(r,s)}(x,y)|A_{T(p,q)}(x,y)$,
 from the formula (\ref{Apoly of T(p,q)}), we have  $rs=pq$.
From the condition $$\D_{T(r,s)}(t)=\frac{(t^{rs}-1)(t-1)}{(t^r-1)(t^s-1)}\left|
\D_{T(p,q)}(t)=\frac{(t^{pq}-1)(t-1)}{(t^p-1)(t^q-1)}\right.,$$
we have \begin{equation}\label{pqrs}
(t^p-1)(t^q-1)|(t^r-1)(t^s-1).
\end{equation}
 Now if  $T(r,s)\ne T(p,q)$, then either $q>s$ or $|p|>|r|$.
If $q>s$, then from (\ref{pqrs}) we must have $q|r$.
By our convention for parameterizing torus knots,
$|p|>q\geq 2$. Thus we must also have $p|r$.
But $p$ and $q$ are relatively prime,  we have $pq|r$, which contradicts
the early conclusion $rs=pq$.
If $|p|>|r|$, again by our convention
$|r|>s\geq 2$, we see that $(t^p-1)$ does not divide
$(t^r-1)(t^s-1)$, and so  (\ref{pqrs}) cannot hold.
This contradiction completes the proof of the lemma.
\qed

By Lemma~\ref{torus knot case}, we see that if $K$ is a torus knot, then
$K=T(p,q)$.

We are going to show that it is impossible
for $K$ to be a satellite knot, which will consist of the rest of the proof of
Theorem~\ref{main1}.
Suppose that $K$ is a satellite knot. We need to derive    a contradiction from
this assumption.
Let  $C$ and $P$ be a pair of associated companion knot and pattern knot
  to $K$, and let
 $w$ be the winding number of $P$ in its defining solid torus $V$ (recall
  the definition of a satellite knot given in Section~\ref{Apoly}).
As $T(p,q)$ is a fibred knot,
$K$ is also fibred.
According to  \cite[Corollary~4.15 and Proposition~8.23]{BurdeZieschang},
each of $C$ and $P$ is a fibred knot in $S^3$,  $w\geq 1$,
and the Alexander polynomials of these knots satisfy the equality
\begin{equation}\label{Alexander poly}
\D_\sk(t)=\D_\sc(t^w)\D_\sp(t).
\end{equation}

We may choose  $C$ such that $C$ is itself not a satellite knot, and thus is
either a hyperbolic knot or a torus knot.

\begin{lem}
The companion knot  $C$ cannot be a hyperbolic knot.
\end{lem}

\pf Suppose that $C$ is hyperbolic. Then $A_\sc(\overline x,\overline y)$ contains
 a balanced-irreducible factor $f_\sc(\overline x,\overline y)$ whose Newton polygon
 detects at least two distinct boundary slopes of $C$.
 As $w\geq 1$, by Proposition~\ref{non-zero winding}, $f_\sc(\overline x,\overline y)$
 extends to a balanced factor $f_\sk(x,y)$ of $A_\sk(x,y)=A_{T(p,q)}(x,y)$.
 Moreover from the relation (\ref{x and x bar}) we see that
 the Newton polygon of $f_\sk(x,y)$ detects at least two distinct boundary slopes of $K$.
 But clearly the Newton polygon of $A_\sk(x,y)=A_{T(p,q)}(x,y)$ only detects one
 boundary slope. We arrive at a contradiction.    \qed

So $C=T(r,s)$ is a torus knot.

\begin{lem}\label{nontrivial pattern}
The pattern knot $P$ of $K$ cannot be the unknot.
\end{lem}

\pf  Suppose otherwise that $P$ is the unknot.
Then as noted in \cite{HMS} the winding number $w$ of $P$ in its
defining solid torus $V$
is larger than $1$.
Equality (\ref{Alexander poly}) becomes
$\D_\sk(t)=\D_\sc(t^w)$ for some integer $w>1$.
On the other hand it is easy to check that the degrees of the leading term and the second term of $\D_{T(p,q)}(t)$ differ by $1$ and thus $\D_\sk(t)=\D_{T(p,q)}(t)$ cannot be of the form
$\D_\sc(t^w)$, $w>1$.
This contradiction completes the proof.
\qed

If $w=1$, then by Proposition ~\ref{non-zero winding},
$A_\sc(x,y)=A_{T(r,s)}(x,y)$ divides
$A_\sk(x,y)=A_{T(p,q)}(x,y)$, and by (\ref{Alexander poly})
$\D_\sc(t)=\D_{T(r,s)}(t)$ divides
$\D_\sk(t)=\D_{T(p,q)}(t)$. Hence
by Lemma~\ref{torus knot case} we have $C=T(r,s)=T(p,q)$.
By  Lemma~\ref{nontrivial pattern},
  $P$  is a non-trivial knot.
Hence from formula (\ref{Alexander poly}), we see that the genus of
$C=T(p,q)$ is  less than that of $K$.
But the genus of $K$ is equal to that of $T(p,q)$. We derive a contradiction.

Hence $w>1$.
By Lemma~\ref{P-factor}, $A_\sp(x,y)$ divides $A_\sk(x,y)$.
Now if $P$ is itself a satellite knot
with its own companion knot $C_1$ and pattern knot $P_1$.
Then again each of $C_1$ and $P_1$ is a fibred knot and the winding number
$w_1$ of $P_1$ with respect to $C_1$ is larger than zero.
Arguing as above, we see that $C_1$ may be assumed to be a torus knot
and that $w_1>1$.
Also we have $$\D_\sk(t)=\D_\sc(t^w)\D_{\sc_1}(t^{w_1})\D_{\sp_1}(t)$$
from which we see that  $P_1$ cannot be the trivial knot just as in the proof of
Lemma~\ref{nontrivial pattern}.

So after a finitely many such steps (the process  must terminate by
\cite{Soma}),  we end up with a pattern knot $P_m$ for $P_{m-1}$ such that $P_m$
is nontrivial but  is no longer a satellite knot.
Thus  $P_m$   is either a hyperbolic knot or a torus knot.
By Lemma~\ref{non-hyperbolic} and Lemma~\ref{P-factor},  $P_m$ cannot be hyperbolic.
Hence $P_m$ is a torus knot. Again because $A_{\sp_m}(x,y)$ divides
$A_\sk(x,y)=A_{T(p,q)}(x,y)$ and $\D_{\sp_m}(t)$ divides
$\D_\sk(t)=\D_{T(p,q)}(t)$, we  have
$P_m=T(p,q)$ by Lemma~\ref{torus knot case}.
But once again  we would have  $g(T(p,q))<g(T(p,q))$.
This gives a final contradiction.


\section{The  knots $k(l,m,n,p)$}\label{topological properties of J}

In \cite{EM1}  a family of hyperbolic knots $k(l,m,n,p)$ in $S^3$ (where at least one of $p$ and $n$ has to be zero)
 was constructed such that each knot in the family admits one (and only one)  half-integral toroidal surgery.
To be hyperbolic, the following restrictions on the values for $l,m,n,p$
are imposed:

 \begin{equation}\label{eq:lmCondition}
 \begin{array}{l}
 \mbox{If $p=0$, then $l\ne 0, \pm1$, $m\ne 0$, $(l,m)\ne (2,1), (-2,-1)$, $(m,n)\ne (1,0),
 (-1,1)$};\\
 \mbox{If $n=0$, then $l\ne 0, \pm1$, $m\ne 0,1$, $(l,m,p)\ne (-2,-1,0)$, $(2,2,1)$}.
 \end{array}
 \end{equation}

From now on we assume that any given $k(l,m,n,p)$
is hyperbolic, i.e. $l,m,n,p$ satisfy the above restrictions.

The half-integral  toroidal slope $r=r(l,m,n,p)$ of $k(l,m,n,p)$  was explicitly computed in \cite[Proposition~5.3]{EM2} as given below:
\begin{equation}\label{eq:rSlope}
r=\left\{
\begin{array}{ll}
l(2m-1)(1-lm)+n(2lm-1)^2-\frac12,\quad&\text{when }p=0;\\
l(2m-1)(1-lm)+p(2lm-l-1)^2-\frac12,\quad&\text{when }n=0.
\end{array}
\right.
\end{equation}

It turns out that $k(l,m,n,p)$ are the only hyperbolic knots in $S^3$
which admit non-integral toroidal surgeries.

\begin{thm}\label{GL2}{\rm (\cite{GL2})}
If  a hyperbolic knot $K$ in $S^3$ admits
a non-integral toroidal surgery, then  $K$ is one of the knots $k(l,m,n,p)$.
\end{thm}

\begin{prop}\label{prop:Dupl}
The knots $k(l,m,n,p)$ have the following properties:
\newline(a) $k(l,m,n,0)$ is the mirror image of  $k(-l,-m,1-n,0)$.
\newline(b) $k(l,m,0,p)$ is the mirror image of  $k(-l,1-m,0,1-p)$.
\newline(c) $k(l,\pm1,n,0)=k(-l\pm1,\pm1,n,0)$.
\newline(d) $k(2,-1,n,0)=k(-3,-1,n,0)=k(2,2,0,n)$.
\end{prop}
\begin{proof}
This follows from \cite[Proposition~1.4]{EM1}.
\end{proof}

Explicit closed braid presentations for  the knots $k(l,m,n,p)$
are given in \cite{EM2}. Figure~\ref{fig:Braid} shows the braids whose closure are $k(l,m,n,p)$. The left two pictures are a reproduction of \cite[Fig.~12]{EM2}, but the right two pictures are different from \cite[Fig.~13]{EM2}. Here an arc with label $s$ means $s$ parallel strands, and a box with label $t$ means $t$ positive full-twists when $t>0$ and $|t|$ negative full-twists when $t<0$.
We only give the picture for the case  $l>0$, since the  case $l<0$ can be treated by applying Proposition~\ref{prop:Dupl}.

\begin{figure}[!ht]
\begin{center}
\begin{picture}(375,167)
\put(0,0){\scalebox{0.55}{\includegraphics*{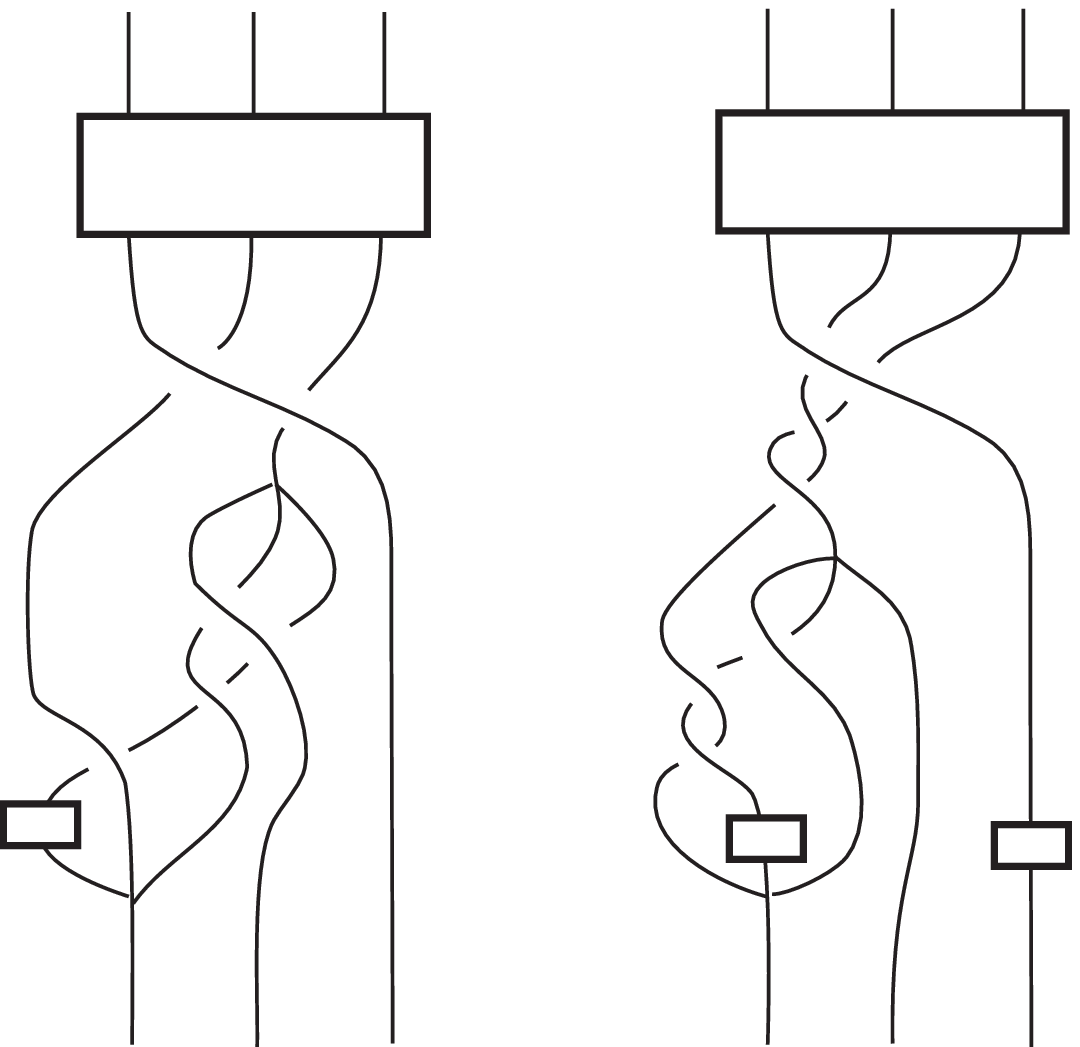}}}

\put(19,0){
\begin{turn}{90}
$\scriptscriptstyle ml-1$
\end{turn}
}

\put(42,3){$\scriptscriptstyle 1$}

\put(38,135){$n$}

\put(60,0){
\begin{turn}{90}
$\scriptscriptstyle ml-1$
\end{turn}
}

\put(25,25){
\begin{turn}{45}
$\scriptscriptstyle l-1$
\end{turn}
}

\put(2,34){$\scriptscriptstyle -1$}

\put(6,64){$\scriptscriptstyle 1$}

\put(47,58){
\begin{turn}{78}
$\scriptscriptstyle ml-l-1$
\end{turn}
}

\put(15,-10){$p=0,m>0$}

\put(120,0){
\begin{turn}{90}
$\scriptscriptstyle -ml$
\end{turn}
}

\put(143,3){$\scriptscriptstyle 1$}

\put(161,0){
\begin{turn}{90}
$\scriptscriptstyle -ml$
\end{turn}
}

\put(102,30){$\scriptscriptstyle 1$}

\put(122,18){
\begin{turn}{50}
$\scriptscriptstyle -ml-l$
\end{turn}
}

\put(123,62){
\begin{turn}{50}
$\scriptscriptstyle l-1$
\end{turn}
}

\put(141,135){$n$}

\put(117,32){$\scriptscriptstyle -1$}

\put(159,30){$\scriptscriptstyle -1$}

\put(116,-10){$p=0,m<0$}

\put(220,1){\scalebox{0.55}{\includegraphics*{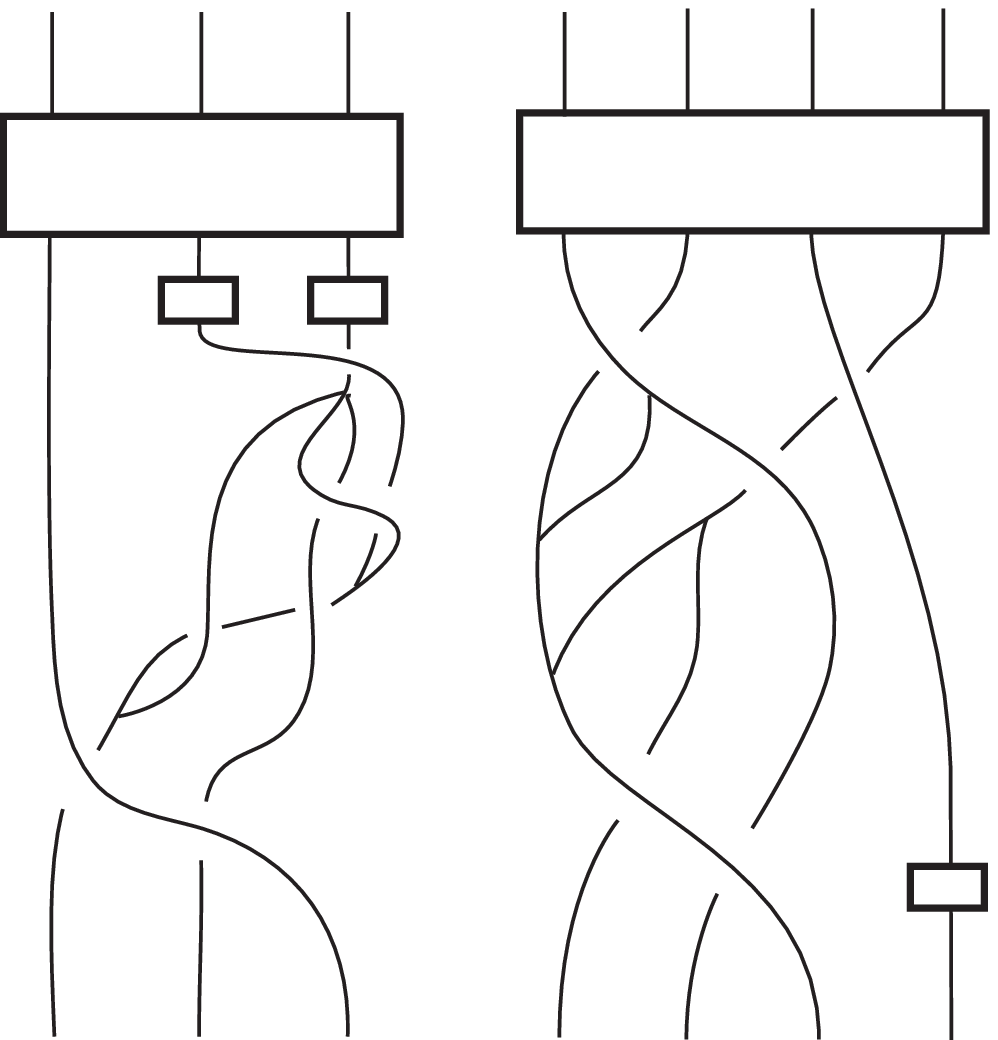}}}

\put(249,135){$p$}

\put(226,3){
\begin{turn}{90}
$\scriptscriptstyle ml-l$
\end{turn}
}

\put(250,3){
\begin{turn}{90}
$\scriptscriptstyle l-1$
\end{turn}
}

\put(282,93){
\begin{turn}{87}
$\scriptscriptstyle l-1$
\end{turn}
}

\put(273,3){
\begin{turn}{90}
$\scriptscriptstyle ml-l$
\end{turn}
}

\put(243,84){
\begin{turn}{60}
$\scriptscriptstyle ml-2l$
\end{turn}
}

\put(260,89){
\begin{turn}{0}
$\scriptscriptstyle 1$
\end{turn}
}

\put(256,69){
\begin{turn}{0}
$\scriptscriptstyle l$
\end{turn}
}

\put(270,117){$\scriptscriptstyle -1$}

\put(247,117){$\scriptscriptstyle -1$}

\put(220,-10){$n=0,m>0$}


\put(337,135){$p$}

\put(367,24){$\scriptscriptstyle -1$}

\put(311,10){$\scriptscriptstyle l$}

\put(331,10){$\scriptscriptstyle 1$}

\put(348,2){
\begin{turn}{90}
$\scriptscriptstyle -ml$
\end{turn}
}

\put(369,2){
\begin{turn}{90}
$\scriptscriptstyle -ml$
\end{turn}
}

\put(312,84){
\begin{turn}{45}
$\scriptscriptstyle l-1$
\end{turn}
}

\put(307,58){
\begin{turn}{50}
$\scriptscriptstyle -ml-l$
\end{turn}
}

\put(314,-10){$n=0,m<0$}

\end{picture}
\end{center}
 \caption{The braid whose closure is $k(l,m,n,p)$, where $l>0$. }\label{fig:Braid}
\end{figure}

\begin{prop}\label{prop:Genus}
Suppose that $l>0$, let
$$N=\left\{\begin{array}{ll}
2ml-1, &\text{if }p=0, n\ne0, m>0,\\
-2ml+1, &\text{if }p=0, n\ne0, m<0,\\
2ml-l-1, &\text{if } n=0, m>0,\\
-2ml+l+1, &\text{if } n=0, m<0.
\end{array}\right.
$$
Then the genus of $k(l,m,n,0)$ is
$$g=|n|\frac{N(N-1)}2+\left\{\begin{array}{ll}
m^2l^2-m\frac{l(l+5)}2+l+1, &\text{if }m>0, n\le0,\\
-m^2l^2+m\frac{l(l+1)}2-l+1, &\text{if }m>0, n>0,\\
m^2l^2-m\frac{l(l-1)}2, &\text{if } m<0, n\le0,\\
-m^2l^2+m\frac{l(l+3)}2, &\text{if } m<0, n>0,
\end{array}\right.
$$
and the genus of $k(l,m,0,p)$ is
$$g=|p|\frac{N(N-1)}2+\left\{\begin{array}{ll}
m^2l^2-m\frac{l(l+5)}2+l+1, &\text{if }m>0, p\le0,\\
-m^2l^2+m\frac{l(l+1)}2+1, &\text{if }m>0, p>0,\\
m^2l^2-m\frac{l(l-1)}2, &\text{if } m<0, p\le0,\\
-m^2l^2+m\frac{l(l+3)}2-l, &\text{if } m<0, p>0.
\end{array}\right.
$$
\end{prop}
\begin{proof}
In \cite{EM2}, it is noted that $k(l,m,n,p)$ is the closure of a positive or negative braid. Hence it is fibered, and the genus can be computed by the formula $$g(k)=\frac{C-N+1}2,$$
where $C$ is the crossing number in the positive or negative braid, and $N$ is the braid index.

When $p=0,m>0,n\le0$, the braid is a negative braid, $N=2ml-1$,
$$C=-n(2ml-1)(2ml-2)+ml(ml-1)+ml-2+l(ml-l-1)+(ml-l-1)(ml-l-2),$$
so
$$g=-n(2ml-1)(ml-1)+m^2l^2-m\frac{l(l+5)}2+l+1.$$

When
$p=0,m>0,n>0$, we can cancel all the negative crossings in the braid to get a positive braid of index $N=2ml-1$. We have
$$C=n(2ml-1)(2ml-2)-\big(ml(ml-1)+ml-2+l(ml-l-1)+(ml-l-1)(ml-l-2)\big),$$
so
$$g=n(2ml-1)(ml-1)-m^2l^2+m\frac{l(l+1)}2-l+1.$$

The computation for other cases are similar.
\end{proof}

Using Proposition~\ref{prop:Genus} and Proposition~\ref{prop:Dupl} (a)(b), we can compute the genus of $k(l,m,n,p)$ when $l<0$.
For example, when $p\le0$,
the genus of $k(l,m,0,p)$ is
\begin{equation}\label{eq:Genus0p}
g=\left\{\begin{array}{ll}
-p\frac{(2ml-l-1)(2ml-l-2)}2+m^2l^2-m\frac{l(l+5)}2+l+1, &\text{if } l>0, m>0,\\
-p\frac{(-2ml+l+1)(-2ml+l)}2+m^2l^2-m\frac{l(l-1)}2, &\text{if } l>0, m<0,\\
-p\frac{(-2ml+l+1)(-2ml+l)}2+m^2l^2-m\frac{l(l-1)}2, &\text{if } l<0, m>0,\\
-p\frac{(2ml-l-1)(2ml-l-2)}2+m^2l^2-m\frac{l(l+5)}2+l+2, &\text{if } l<0, m<0.
\end{array}\right.
\end{equation}

The $r$-surgery on the knot $J=k(l,m,n,p)$  was explicitly given in \cite{EM1}
which is the double branched cover of $S^3$ with the branched set in $S^3$ being a link
shown in Figure~\ref{tangle1}. From the tangle decomposition of the branched link
one can see that $M_\sj(r)$ is a graph manifold obtained by gluing two
Seifert fibred spaces, each over a disk with two cone points, together along
their torus boundaries.
For our purpose, we need to give a more detailed description of the graph
manifold $M_\sj(r)$ as follows.

 Let $(B,t)$ denote  a two string tangle, i.e. $B$ is a $3$-ball and $t$ is a
 pair of disjoint properly embedded arcs in $B$. Here we may assume that $B$ is
 the unit $3$-ball in the $xyz$-space $\mathbb R^3$ (with the $xy$-plane horizontal)
 and that the four endpoints of $t$ lies in the lines $z=y$, $x=0$ and $z=-y,x=0$.
 Let $D$ be the unit disk in $B$ which is the intersection of $B$ with
 the $yz$-plane. Then the four endpoints of $t$ divides $\p D$ into four arcs,
 naturally named the east, west, north, south arcs.
The {\it denominator closure} of $(B,t)$ is the link in $S^3$ obtained by capping off
$t$ with the east and west arcs, and the {\it numerator closure} of $(B,t)$ is the link in $S^3$ obtained by capping off
$t$ with the north and south arcs.

Let $(B_i,t_i)$, $i=1,2$, be the two tangles shown in Figure~\ref{tangle1}
and let $X_i$ be the double branched cover of $(B_i,t_i)$.
The denominator closure of $(B_1,t_1)$ is the twisted knot of type $(2,p)$
(which is the trivial knot when $p=0,1$, the trefoil knot when $p=-1$)
and therefore the double branched cover of $S^3$ over the link
is the lens space of order $|2p-1|$.
The numerator closure of $(B_1,t_1)$ gives a composite link in $S^3$
and in fact the composition  of two nontrivial rational links
corresponding  to the rational numbers $-l$ and $(-lm(2p-1)+pl+2p-1)/(-m(2p-1)+p)$.

\begin{figure}[!ht]
\centerline{\includegraphics{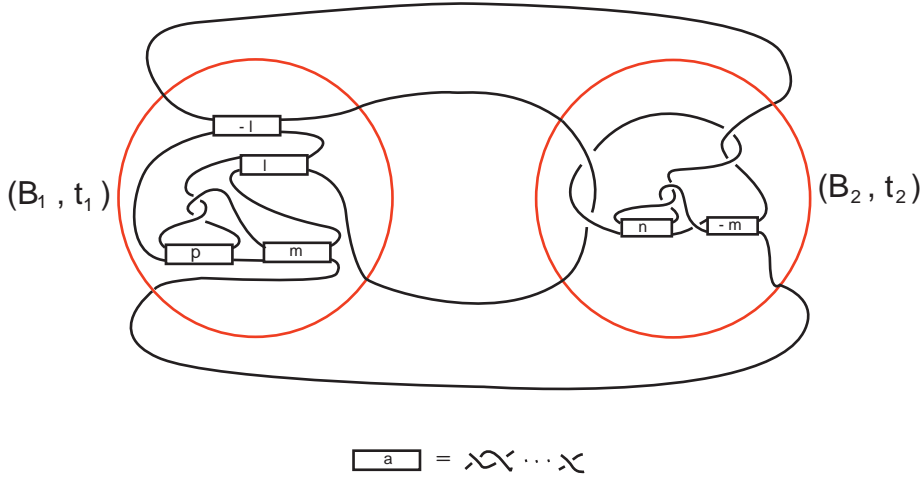}} \caption{The tangle decomposition of the branched link in $S^3$ for the half-integral toroidal surgery}\label{tangle1}
\end{figure}

The double branched cover of the east arc
(as well as the west arc) is a simple closed essential curve in $\p X_1$, which we denote by $\mu_1$,
with which Dehn filling of $X_1$ is a lens space of order $|2p-1|$.
 The double branched cover of the north arc (as well as the south arc)
 is a Seifert fiber of $X_1$, which we denote by $\s_1$, with which Dehn
 filling of $X_1$ is a connected sum of two nontrivial lens spaces of
 orders $|-l|$ and $|-lm(2p-1)+pl+2p-1|$.

 Similarly the numerator  closure of $(B_2,t_2)$ is the twisted knot of type $(2,n)$
and the denominator  closure of $(B_2,t_2)$ is a composite link of
 two nontrivial rational links corresponding to the rational numbers
 $-2$ and $(2(2n-1)(m-1)+4n-1)/(2n-1)(m-1)+n)$.
 So the double branched cover of the north arc of $(B_2,t_2)$
 is a simple closed essential curve in $\p X_2$, denoted $\mu_2$,
with which Dehn filling of $X_2$ is a lens space of order $|2n-1|$, and
 the double branched cover of the west arc
 is a Seifert fiber of $X_2$, denoted by $\s_2$, with which Dehn filling is a connected sum of
 two nontrivial lens spaces of orders $2$ and $|2(2n-1)(m-1)+4n-1|$.

 Finally $M_\sj(r)$ is obtained by gluing $X_1$ and $X_2$ along their boundary tori
 such that $\m_1$ is identified with $\s_2$ and $\s_1$ with $\mu_2$.

Note that when $p=0$ or $1$, $X_1$ is the exterior of a torus knot in $S^3$,
and when $n=0$ or $1$, $X_2$ is the exterior of a $(2,a)$ torus knot in $S^3$.

\begin{lem}\label{only one incomp surface}
Let $J=k(l,m,n,p)$, $r$  the unique half-integral toroidal slope
of $J$ and $M_\sj$ the exterior of $J$.
Up to isotopy, there is a unique closed orientable incompressible surface
in $M_\sj(r)$, which is a torus.
\end{lem}

\pf As discussed above, $M_\sj(r)$ is a graph manifold
with the torus decomposition
$M_\sj(r)=X_1\cup X_2$.
We just need to show that any connected closed orientable incompressible surface $S$ in $M_\sj(r)$ is isotopic to $\p X_1$.
Suppose otherwise that $S$ is not isotopic to $\p X_1$.
As each $X_i$ does not contain closed essential
 surfaces, $S$ must intersect $\p X_1$.
We may assume that $F_i=S\cap X_i$ is incompressible
and boundary incompressible, for each $i=1,2$.
As $X_i$ is Seifert fibred, $F_i$ is either horizontal (i.e. consisting
of Seifert fibres or
vertical (transverse to Seifert fibres) in $X_i$, up to isotopy.
So  the boundary slope of $F_i$ is  either $\s_i$ (when $F_i$ is
horizontal) or is the rational longitude of $X_i$ (when $F_i$ is vertical).
But $\s_i$ is identified with $\mu_j$, for $i=1,2$, $\{i,j\}=\{1,2\}$.
So the boundary slope of $F_i$ must be the rational longitude of
$X_i$ for each $i$.
It follows that $H_1(M_\sj(r);\z)$ is infinite, yielding a contradiction.
\qed

\begin{lem}\label{meridian not bdy slope}
For each  $J=k(l,m,n,p)$, its meridian slope  is not a boundary slope.
\end{lem}

\pf  If the meridian
slope of $J$ is a boundary slope, then
the knot exterior $M_\sj$
contains a connected closed orientable incompressible surface
$S$ of genus larger than one such that $S$ remains incompressible
in any non-integral surgery, by \cite[Theorem 2.0.3]{CGLS}.
But by Lemma~\ref{only one incomp surface},
 $M_\sj(r)$  does not contain any closed orientable incompressible surface
of genus larger than one.
This contradiction completes the proof.
\qed

\begin{figure}[!ht]
\centerline{\includegraphics{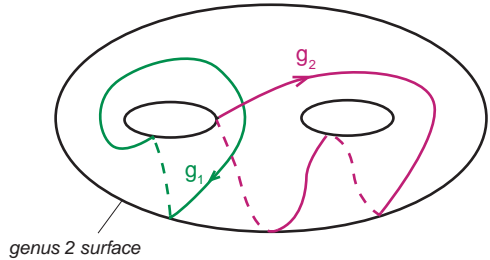}}\caption{The
elements $g_1$ and $g_2$ in $\p H$}\label{genus two surface}
\end{figure}

Lastly in this section we are going to show that
$k(l_*,-1,0,0)$ is a class of small knots.

Let $H$ be a standard genus two handlebody in $S^3$
and let $g_1$ and $g_2$ be oriented loops in the surface $\p H$ as shown
in Figure~\ref{genus two surface}.
 As observed in
\cite[Figure~8]{Berge}, a regular neighborhood of $g_1\cup g_2$ in $\partial H$ is a genus one Seifert surface   of a trefoil knot.
Under the basis $[g_1],[g_2]$, the monodromy of the trefoil is represented by the matrix
$$
\begin{pmatrix}
0 &1\\
-1 &1
\end{pmatrix}.
$$
From \cite[Figure 11]{EM2}, the knot $k(l_*,-1,0,0)$ has a knot diagram as
shown in Figure~\ref{fig (l,-1)} (a), which in turn can be
embedded in  the Seifert surface with the homology class $l_*[g_1]+(l_*+1)[g_2]$
as shown in Figure~\ref{fig (l,-1)} (b).
For any knot $J$ lying in the Seifert surface,
an algorithm to classify all closed essential surfaces in the
knot exterior of $J$ is given in \cite[Theorem 10.1 (3)]{Baker}.
For simplicity, we only state the part that we need:

\begin{figure}[!ht]
\centerline{\includegraphics{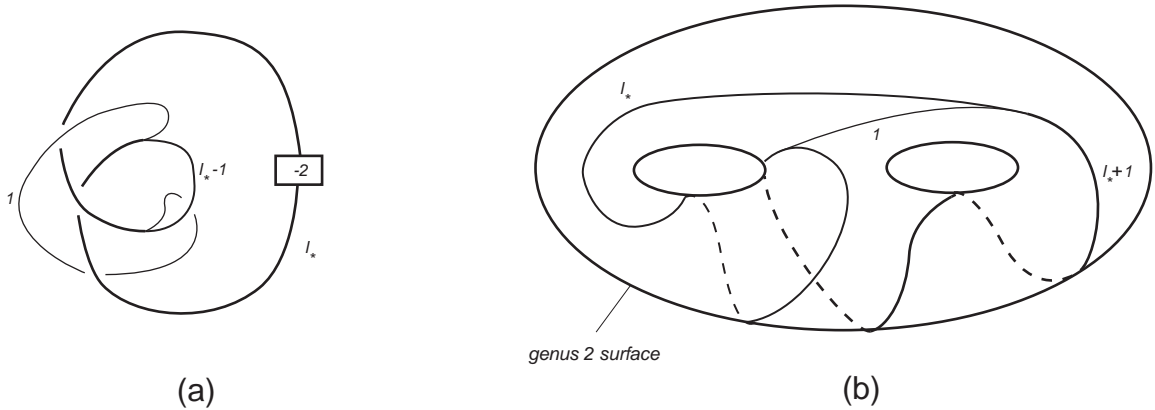}} \caption{The knot
$k(l_*,-1,0,0)$}\label{fig (l,-1)}
\end{figure}

\begin{thm}
Suppose that $J$ is a simple closed  curve on the genus one Seifert surface of a trefoil knot in the homology class $a_1[g_1]+a_2[g_2]$.
The slope $\frac {a_1}{a_2}$ has continued fraction expansion
$$\frac {a_1}{a_2}=[b_1,\dots,b_k]=b_1-\frac1{\displaystyle b_2-\frac1{\displaystyle\ddots-\frac1{b_k}}},$$
where the coefficients alternate signs, $b_i\ne0$ when $i\ge2$, and $|b_k|\ge2$. If $b_1=0$ and $b_2=-1$, then every closed essential surface in
the complement of $L$ corresponds to a solution of the following equation:
\begin{equation}\label{eq:EssSurf}
0=\sum_{i\in I}-b_i+\sum_{j\in J}b_j+
\left\{
\begin{array}{ll}
0, &\text{if }3\in J,\\
-1, &\text{otherwise,}
\end{array}
\right.
\end{equation}
where $I$ and $J$ are subsets of $\{3, . . . , k\}$ each not containing consecutive integers and $3\notin I\cap J$.
\end{thm}

\begin{lem}\label{lem:Small}
Each $k(l_*,-1,0,0)$ is a small knot. Namely, the exterior  of $k(l_*,-1,0,0)$ contains no closed essential surfaces.
\end{lem}
\begin{proof}
In this case $a_1=l_*, a_2=l_*+1$, $\frac{a_1}{a_2}=[0,-1,l_*]$. Clearly, Equation (\ref{eq:EssSurf}) has no solution. So $k(l_*,-1,0,0)$ is small.
\end{proof}


\section{Information on  A-polynomials of $k(l,m,n,p)$}\label{A-poly of J}

To obtain explicit expressions for the $A$-polynomials of the knots
$k(l,m, n,p)$ could be a very tough task.
But we can obtain some useful info about
the  $A$-polynomials of these knots
without computing them explicitly.

\begin{lem}\label{lem:AbelRep}
Suppose a graph manifold $W$ is obtained by gluing two torus knot exteriors
$X_1,X_2$ together, such that the meridian of $X_i$ is glued to the Seifert
fiber of $X_{i+1}$, $i=1,2$, where $X_3=X_1$. Then $\pi_1(W)$ has no
non-cyclic $SL_2(\mathbb C)$ representations.
\end{lem}

To prove this lemma,
we will use the following well-known fact whose proof is elementary.

\begin{lem}\label{lem:SL2comm}
Suppose that $\tensor A,\tensor B\in SL_2(\mathbb C)$ are two commuting matrices, $\tensor A\ne\pm  I$.
\newline(i) If there exits $\tensor P\in SL_2(\mathbb C)$ such that
\begin{equation}\label{eq:Diag}
\tensor P\tensor A\tensor P^{-1}=\begin{pmatrix}
\lambda & 0\\
0 &\lambda^{-1}
\end{pmatrix},\qquad\text{for some }  \lambda\in\mathbb C\setminus\{0,\pm1\},
\end{equation}
then
$$\tensor P\tensor B\tensor P^{-1}=\begin{pmatrix}
\mu & 0\\
0 &\mu^{-1}
\end{pmatrix},\qquad\text{for some } \mu\in\mathbb C\setminus\{0\};
$$
\newline(ii) If there exits $\tensor P\in SL_2(\mathbb C)$ such that
\begin{equation}\label{eq:Para}
\tensor P\tensor A\tensor P^{-1}=\pm\begin{pmatrix}
1 & a\\
0 &1
\end{pmatrix},\qquad\text{for some }  a\in\mathbb C\setminus\{0\},
\end{equation}
then
$$\tensor P\tensor B\tensor P^{-1}=\pm\begin{pmatrix}
1 & b\\
0 &1
\end{pmatrix},\qquad\text{for some }  b\in\mathbb C.
$$
\end{lem}

\begin{proof}[Proof of Lemma~\ref{lem:AbelRep}]
As $W$ has cyclic homology group, it is equivalent to show
that every $SL_2(\c)$-representation of $\pi_1(W)$ is abelian.

We choose a base point of $W$ on the common boundary torus $T$ of $X_1,X_2$, then $\pi_1(T), \pi_1(X_1),\pi_1(X_2)$ are naturally the subgroups of $\pi_1(W)$.
Suppose that $\rho\co\pi_1(W)\to SL_2(\mathbb C)$ is a representation.

Let $f_i\in\pi_1(T)$ represent the Seifert fiber of $X_i$, $i=1,2$. We first consider the case that one of ${\rho}(f_1),{\rho}(f_2)$, say ${\rho}(f_1)$, is in $\{\pm I\}$ which is the center of $SL_2(\mathbb C)$.
Since $f_1$ represents the meridian of $X_2$, $f_1$ normally generates $\pi_1(X_2)$, hence ${\rho}(\pi_1(X_2))$ is contained in $\{\pm  I\}$. Since $f_2\in\pi_1(X_2)$, we see that ${\rho}(f_2)$ is also  in $\{\pm  I\}$, and thus for the same reason as just given ${\rho}(\pi_1(X_1))$ is in $\{\pm  I\}$. Hence $\rho(\pi(W))$ is contained  in $\{\pm I\}$.

Now suppose that neither  of ${\rho}(f_1),{\rho}(f_2)$ is in $\{\pm  I\}$, then there exists $\tensor P\in SL_2(\mathbb C)$ such that $\tensor P{\rho}(f_1)\tensor P^{-1}$ is in the form of either (\ref{eq:Diag}) or (\ref{eq:Para}).  Since $f_1$ is in the center of $\pi_1(X_1)$, by Lemma~\ref{lem:SL2comm}, $\tensor P\rho(\pi_1(X_1))\tensor P^{-1}$ (including $\tensor P{\rho}(f_2)\tensor P^{-1}$) is contained in an abelian subgroup $A$ of $SL_2(\mathbb C)$ ($A$ is either the set of diagonal matrices or the set of upper triangular trace $\pm 2$ matrices). Since $f_2$ is in the center of $\pi_1(X_2)$, again by Lemma~\ref{lem:SL2comm}, $\tensor P\rho(\pi_1(X_2))\tensor P^{-1}$ is also contained in $A$.
Hence ${\rho}(\pi_1(W))$, being generated by ${\rho}(\pi_1(X_1)),{\rho}(\pi_1(X_2))$, is an abelian group.
\end{proof}

\begin{cor}\label{prop:AbelRep}
The fundamental group of the half-integral  toroidal surgery on
$J=k(l,m,0,0)$ has no non-cyclic $SL_2(\mathbb C)$
representations and has no non-cyclic $PSL_2(\c)$ representations.
\end{cor}

\pf By the discussion preceding Lemma~\ref{only one incomp surface},
the half-integral toroidal surgery, $M_\sj(r)$, is a graph manifold
satisfying the conditions of Lemma~\ref{lem:AbelRep}.
Hence $M_\sj(r)$ has no non-cyclic $SL_2(\c)$ representations.
The manifold cannot have non-cyclic $PSL_2(\c)$ representations either
since every $PSL_2(\c)$ representation of $M_\sj(r)$ lifts to
a $SL_2(\c)$-representation because the manifold has odd cyclic
first homology.
\qed

\begin{lem}\label{r is a vertex}Let $M_\sj$ be the knot exterior of any given hyperbolic  knot $J=k(l,m,0,0)$.
Let $X_0$ be any norm curve   in $X(M_\sj)$ and $B_0$ the
norm polygon determined by $X_0$. Then the half-integral toroidal slope
$r=d/2$ of $J$ is associated to a vertex of $B_0$ as described in Theorem~\ref{properties of a norm curve} (3), i.e.  $2/d$ is
the slope of a vertex of $B_0$ in the $xy$-plane $H_1(\p M_\sj; \mathbb R)$.
\end{lem}

\pf Suppose otherwise that $r$ is not associated to a vertex  of $B_0$.
As the meridian slope of $M_\sj$ is not a boundary slope by Lemma~\ref{meridian not bdy slope},
it follows from Theorem~\ref{properties of a norm curve} (4) and (5)
that $\m$ is contained in $\p B_0$ but $r$ is not, which means
that $Z_v(\tilde f_r)>Z_v(\tilde f_\m)$ for some point $v\in \tilde X_0$.
As $M_\sj(r)$ has no noncyclic  representations by Lemma~\ref{prop:AbelRep},
the point $v$ cannot be a regular point of $\tilde X_0$ (by \cite[Proposition 1.5.2]{CGLS} or \cite[Proposition 4.8]
{BZ4}).
So $v$ is an ideal point of $\tilde X_0$.
As $r$ is not a slope  associated to any vertex of $B_0$,
$\tilde f_\a(v)$ is finite for every class
$\a$ in $H_1(\p M_\sj; \z)$.
Now we may apply \cite[Proposition 4.12]{BZ4} to see that $M_\sj$ contains a
closed essential surface $S$ such that if $S$ compresses in
$M_\sj(r)$ and $M_\sj(\a)$, then $\D(r,\a)\leq 1$.
By
Lemma~\ref{only one incomp surface}, $S$ must  compress in $M_\sj(r)$
and of course $S$ compresses in $M(\m)$.
But $\D(r,\m)=2$. We arrive at a contradiction.
\qed

\begin{lem}\label{J has only 3 roots}Let $A_\sj(x,y)$ be the $A$-polynomial of
 any given hyperbolic knot $J=k(l,m,0,0)$.
 Let $r=d/2$ be the half-integral toroidal slope of $J$.
 If $(x_0,y_0)$ is a solution
of the system
$$\left\{\begin{array}{ll}A_\sj(x,y)&=0,\\
x^dy^2-1&=0,\end{array}\right.
$$
then $x_0\in \{0, 1,-1\}$.
 \end{lem}

\pf Suppose otherwise that $x_0\notin \{0, 1,-1\}$.
Then by the constructional definition of the $A$-polynomial,
there is a component  $X_1$ in $X^*(M_\sj)$ which
contributes a factor $f_0(x,y)$ in $A_\sj(x,y)$
such that $(x_0, y_0)$ is a solution of
$$\left\{\begin{array}{ll}f_0(x,y)&=0,\\
x^dy^2-1&=0,\end{array}\right.
$$
Let $Y_0$ be the Zariski closure of $\widehat{i}_*(X_1)$
in $X(\p M_\sj)$. We knew $Y_0$ is an irreducible curve.
We may find an irreducible curve $X_0$ in $X_1$
such that $Y_0$ is also the  Zariski closure of $\widehat{i}_*(X_0)$.
Now it also follows from the constructional definition of the $A$-polynomial that
there is a convergent sequence of regular points $\{v_i\}\subset \tilde X_0$
such that $\tilde f_r(v_i)\ra 0$ and $\tilde f_\m(v_i)
\ra (x_0+x_0^{-1})^2-4$, i.e.
if  $v$ is the limit point of $v_i$ in $\tilde X_0$,
then $\tilde f_r(v)= 0$ and $\tilde f_\m(v)
= (x_0+x_0^{-1})^2-4\ne 0$.

Note that $\tilde f_r$ is not constant on $X_0$.
For otherwise $X_0$ would be a semi-norm curve with $r$ as the
associated slope and would have
a contradiction with Theorem~\ref{properties of a semi-norm curve} (3)
(since by Lemma~\ref{meridian not bdy slope} $\m$ is not a boundary slope).

So we have $Z_v(\tilde f_r)>Z_v(\tilde f_\m)=0$.
Again  $v$ cannot be a  regular point of $\tilde X_0$
due to  Lemma~\ref{prop:AbelRep}.
So $v$ is an  ideal point  of $\tilde X_0$
such that $f_r(v)=0$ and $f_\m(v)\ne 0$ is finite.
We get a contradiction with \cite[Prposition~4.12]{BZ4} as
in the proof of Lemma~\ref{r is a vertex}.
\qed

\begin{lem}\label{nonhyp surgery}
Let $K$ be a hyperbolic knot in $S^3$. For a given slope $p/q$, if every solution $(x_0,y_0)$
of the system
of equations $$\left\{\begin{array}{ll}A_\sk(x,y)&=0,\\
x^py^q-1&=0,\end{array}\right.
$$ has  $x_0\in\{1,-1,0\}$, then
$M_\sk(p/q)$ is not a hyperbolic manifold.
\end{lem}

\pf Some of the ideas for the proof come from  \cite{BZ3}.
Suppose otherwise that $M_\sk(p/q)$ is a hyperbolic $3$-manifold.
Then $\pi_1(M_\sk(p/q))$ has a discrete faithful representation $\bar \r_0$ into
$PSL_2(\c)$. By Thurston (\cite[Proposition 3.1.1]{CS}), this representation  can be lifted to
a $SL_2(\c)$-representation $\r_0$. It follows from the Mostow rigidity
that the character $\chi_{\r_0}$ of $\r_0$ is an isolated point
in $X(M_\sk(p/q))$.  Note that $\r_0$ can be considered as an element
 in $R(M_\sk)$ and $\chi_{\r_0}$ can be considered as an element in
  $X(M_\sk)$ since $R(M_\sk(p/q))$ embeds in $R(M_\sk)$ and  $X(M_\sk(p/q))$ embeds in $X(M_\sk)$.
 Of course
we have $\r_0(\m^{p}\l^q)=I$ but $\r_0(\m)\ne \pm I$.
Let $X_0$ be a component of $X(M_\sk)$ which contains $\chi_{\r_0}$.
By Thurston (\cite[Proposition~3.2.1]{CS}), $X_0$ is positive dimensional.

\begin{claim}\label{claim:non-constant}
 The function $f_{\m^p\l^q}$ is
non-constant on  $X_0$.\end{claim}

Suppose otherwise. Then $f_{\m^p\l^q}$ is constantly zero on $X_0$ since
$f_{\m^p\l^q}(\chi_{\r_0})=0$.
So for every $\chi_\r\in X_0$, $\r(\m^p\l^q)$ is either $I$ or $-I$
or is a parabolic element.
Let $X_1$ be an irreducible curve in $X_0$ which contains the point $\chi_{\r_0}$.
For a generic point $\chi_\r\in X_1$, $\r(\m^p\l^q)$ cannot be a parabolic element
since otherwise $\r(\m)$ is either $I$ or $-I$ or parabolic for all
$\chi_\r\in X_1$ and this happens in particular at the point $\chi_{\r_0}$, yielding
a contradiction.
For a generic point $\chi_\r\in X_1$, $\r(\m^p\l^q)$ cannot be $-I$ either
for otherwise by continuity, $\r(\m^p\l^q)=-I$
for every point $\chi_\r\in X_1$, and this happens in particular
at the point $\chi_{\r_0}$, yielding
another contradiction.
So for a generic point $\chi_\r\in X_1$, $\r(\m^p\l^q)=I$
and again by continuity, $\r(\m^p\l^q)=I$ for every
point $\chi_\r\in X_1$.
So $X_1$ factors though the $p/q$-surgery on $K$ and becomes a
subvariety of $X(M_\sk(p/q))$. But this contradicts the
fact that $\chi_{\r_0}$ is an isolated point
of $X(M_\sk(p/q))$. The claim is thus proved.

It also follows from the proof of Claim~\ref{claim:non-constant} that $X_0$ is one dimensional. For otherwise there would be a curve $X_1$ in $X_0$ such that
$\chi_{\r_0}\in X_1$ and $f_{\m^p\l^q}$ is constantly zero on $X_1$, which is impossible by the proof of Claim~\ref{claim:non-constant}.

It follows from Claim~\ref{claim:non-constant} that the restriction of $X_0$
in $X(\p M_\sk)$ is one dimensional and thus $X_0\in X^*(M_\sk)$ and contributes
a factor $f_0(x,y)$ to $A_\sk(x,y)$.

We may assume, up to conjugation of $\r_0$, that
$$\r_0(\m)=\left(\begin{array}{cc}x_0&a\\0&x_0^{-1}\end{array}\right),\;\;\;
\r_0(\l)=\left(\begin{array}{cc}y_0&b\\0&y_0^{-1}\end{array}\right).$$
Note that $x_0\ne \pm 1$ (as $\r_0(\m)\ne \pm I$ and cannot be a parabolic element of $SL_2(\c)$).
By the construction of $A_\sk(x,y)$, $(x_0, y_0)$ is a solution
 of  the system
 $$\left\{\begin{array}{ll}f_0(x,y)&=0,\\
x^py^q-1&=0.\end{array}\right.$$
and thus is a solution of
  $$\left\{\begin{array}{ll}A_\sk(x,y)&=0,\\
x^py^q-1&=0.\end{array}\right.$$ We get a contradiction with the assumption of the lemma.
\qed

\begin{prop}\label{same r slope}
If $K\subset S^3$ is a hyperbolic knot whose $A$-polynomial
divides  the $A$-polynomial of $J=k(l,m,0,0)$, then $K$ has the same half-integral toroidal slope as $J$
and thus $K$ is one of $k(l,m,n,p)$.
\end{prop}

\pf Since $K$ is hyperbolic, $X(M_\sk)$ contains a norm curve
component $X_0'$ which contributes a balanced-irreducible factor
$f_0(x,y)$ to $A_\sk(x,y)$ such that the Newton polygon of $f_0(x,y)$
is dual to the norm polygon determined by $X_0'$ by Theorem~\ref{dual polygons}.
By the assumption that $A_\sk(x,y)$ divides $A_\sj(x,y)$,
$f_0(x,y)$ is also a factor of $A_\sj(x,y)$.
Thus there is a curve $X_0$ in a component of $X^*(M_\sj)$
which contributes $f_0(x,y)$ and $X_0$ must be a norm curve
whose norm polygon $B_0$ is dual to the Newton polygon of $f_0(x,y)$.
By Lemma~\ref{r is a vertex}, the half-integral toroidal slope
$r=d/2$ of $J$ is associated to a vertex of $B_0$ and thus
$r=d/2$ is also associated to an edge of the Newton polygon of
$f_0(x,y)$.
Hence $r$ is also a boundary slope of $K$.

Again by the  assumption that $A_\sk(x,y)$ divides $A_\sj(x,y)$,
together with Lemma~\ref{J has only 3 roots}, we see that
if $(x_0,y_0)$ is a solution
of the system
$$\left\{\begin{array}{ll}A_\sk(x,y)&=0,\\
x^dy^2-1&=0,\end{array}\right.
$$
then $x_0\in \{0, 1,-1\}$.
Now applying  Lemma~\ref{nonhyp surgery},
we see that $M_\sk(r)$ is not a hyperbolic manifold.
Applying \cite[Theorem 2.0.3]{CGLS} and \cite{GL}
we see that $M_\sk(r)$ must be a Haken manifold and
thus must be  a toroidal manifold (as it has finite first homology).
Finally    $K$ is one of the  knots $k(l,m,n,p)$
by Theorem~\ref{GL2}.
\qed

\begin{lem}\label{lmn-irr-rep}
The half-integral toroidal $r$-surgery on $J=k(l,m,n,0)$, $n\ne 0,1$,
is a manifold with an irreducible $SL_2(\c)$ representation $\r_0$
whose image contains no parabolic elements.
\end{lem}

\pf We knew from Section~\ref{topological properties of J} that
$M_\sj(r)=X_1\cup X_2$, where  $X_1$ is a $(a,b)$-torus knot
exterior and $X_2$ is Seifert fibred with base orbifold $D^2(2,c)$
for some odd integer $c>1$, such that
 the meridian slope $\m_1$ of $X_1$ is identified with
the Seifert fibre slope $\s_2$ of $X_2$ and the Seifert fiber slope $\s_1$  of $X_1$ is identified with
 a lens space filling slope $\m_2$  of $X_2$ (and the lens space has order $|2n-1|$).

Perhaps it is easier to construct a $PSL_2(\c)$ representation $\bar\r_0$ of $\pi_1(M_\sj(r))$ with the required  properties.
As  $M_\sj(r)$ has zero $\z_2$-homology,
every $PSL_2(\c)$ representation of $\pi_1(M_\sj(r))$ lifts to  a $SL_2(\c)$ representation.

The representation $\bar\r_0$ will send
$\pi_1(X_2)$ to a cyclic group of order $|2n-1|$
which is possible by factoring through $X_2(\m_2)$.
So $\bar\r_0(\m_2)=id$.
We claim $\bar\r_0(\s_2)$ is not the identity element.
For otherwise $\bar\r_0$ factors through the group
$$<x,y; x^2=y^c=1, xy=1>$$
which is the trivial group as $c$ is odd.

Hence the order of $\bar\r_0(\s_2)$ is an odd
number $q>1$ which is a factor of $2n-1$.

On the $X_1$ side, we need to have $\bar\r_0(\s_1)=id$ and
$\bar\r_0(\m_1)=\bar\r_0(\s_2)$  of order $q$.
So $\bar\r_0$ factors through the triangle group
$$<x,y; x^a=y^b=(xy)^q=1>.$$
Such representation exists and can be required to be irreducible.
Also as at least  one of $a$ and $b$ is odd  and $q$ is odd, we may require
the image of $\bar\r_0$ containing no parabolic elements.
In fact the triangle group is either a spherical or a hyperbolic triangle group
and so we may simply choose
 $\bar\r_0$ to be a discrete faithful representation of the triangle group into $SO(3)\subset  PSL_2(\c)$ (when the triangle group is spherical)
or into $PSL_2(\mathbb R)\subset  PSL_2(\c)$ (when the triangle group is hyperbolic), and thus the image
of $\bar\r_0$ has no parabolic elements.
\qed

\begin{lem}\label{lmp-irr-rep}For any given $J=k(l,m,0,p)$,
with $p\ne 0,1$, and with $l$ not divisible by $|2p-1|$,
the half-integral toroidal surgery on $J$
is a manifold with an irreducible $SL_2(\c)$ representation $\r_0$
whose image contains no parabolic elements.
\end{lem}

\pf The proof is similar to
that of Lemma~\ref{lmn-irr-rep}.

We knew $M_\sj(r)=X_1\cup X_2$, where $X_2$ is a $(2,a)$-torus knot
exterior and $X_1$ is Seifert fibred with base orbifold $D^2(|l|,|-lm(2p-1)+pl+2p-1|)$,
such that  the meridian slope $\m_2$  of $X_2$ is identified with
the Seifert fibre slope $\s_1$ of $X_1$ and the Seifert fiber slope $\s_2$ of $X_2$ is identified with
 a lens space filling slope $\m_1$ of $X_1$ (and the lens space has order $|2p-1|$).

As in Lemma~\ref{lmn-irr-rep}, we jus need to  construct a $PSL_2(\c)$ representation $\bar\r_0$ of $\pi_1(M_\sj(r))$ which is irreducible and whose image contains no parabolic elements.

The representation $\bar\r_0$ will send
$\pi_1(X_1)$ to a cyclic group of order $|2p-1|$
which is possible by factoring through $X_1(\m_1)$.
So $\bar\r_0(\m_1)=id$.
We claim $\bar\r_0(\s_1)$ is not the identity element.
For otherwise $\bar\r_0$ factors through the group
$$<x,y; x^l=y^{-lm(2p-1)+pl+2p-1}=1, xy=1>$$
which is a cyclic group of order less than $|2p-1|$
since    $l$ is not divisible
by  $2p-1$ by our assumption.

Hence the order of $\bar\r_0(\s_2)$ is an odd
number $q>1$ which is a factor of $2p-1$.

On the $X_2$ side, we  need to have $\bar\r_0(\s_2)=id$ and
$\bar\r_0(\m_2)=\bar\r_0(\s_1)$  of order $q$.
So $\bar\r_0$ factors through the triangle group
$$<x,y; x^2=y^a=(xy)^q=1>.$$
Such representation exists and can be required to be irreducible.
Also as both $a$ and $q$ are odd, we may require
the image of $\bar\r_0$ containing no parabolic elements.
\qed

\begin{lem}\label{(l,m,n) is different from (l',m')}
The $A$-polynomial of any $J=k(l, m,n,0)$, $n\ne 0,1$,
does not divide the $A$-polynomial of any  $J'=k(l',m',0,0)$.
\end{lem}

\pf Suppose otherwise that $A_\sj(x,y)|A_{\sj'}(x,y)$.
 Then by Proposition~\ref{same r slope}, $J$ and $J'$ have the same half-integral toroidal slope $r=d/2$, $d$ odd.

 Let $\r_0$ be an irreducible representation of
$M_\sj(r)$ provided by Lemma~\ref{lmn-irr-rep}.
Then $\r_0(r)=I$ but $\r_0(\m)\ne \pm I$.
We know that $\chi_{\r_0}$ is contained in a positive
dimensional component $X_1$ of $X(M_\sj)$.
Let $X_0$ be an irreducible curve in $X_1$ containing
$\chi_{\r_0}$.

Claim.  $f_r$ is not constant on $X_0$.

Otherwise $f_r$ is constantly equal to $4$ on $X_0$.
If $f_\mu$ is not a constant on $X_0$, then $X_0$ would be  a semi-norm
with $r$ as its associated boundary slope, which
 is impossible by Theorem~\ref{properties of a semi-norm curve}
 as $\mu$ is not a boundary slope
and $\D(\m,r)=2$.
So $f_\mu$ is a constant not equal to $4$ on $X_0$
since $\r_0(\m)\ne \pm I$ and is not parabolic.
So for  any point $\chi_\r\in X_0$,
$\r(r)$ cannot be  a parabolic element (for otherwise
$\r(\mu)$ is also parabolic and thus $f_\m(\chi_\r)=4$).
So $\r(r)=I$ for any $\chi_\r\in X_0$.
We now get a contradiction with \cite[Proposition 4.10]{BZ4},
which proves the claim.

So $f_r$ is not constant on  $X_0$ which means that the component $X_1\supset X_0$
 belongs to $X^*(M_\sj)$ and thus contributes a factor in $A_\sj(x,y)$.
 Moreover the point  $\chi_{\r_0}$ contributes a root $(x_0,y_0)$
to the system
$$\left\{\begin{array}{l}A_\sj(x,y)=0\\
x^dy^2=1\end{array}\right.$$
 such that $x_0\ne \pm 1, 0$.
As $A_\sj(x,y)|A_{\sj'}(x,y)$,
$(x_0,y_0)$ is also a solution of the system $$\left\{\begin{array}{l}A_{\sj'}(x,y)=0\\
x^dy^2=1,\end{array}\right.$$ which
contradicts Lemma~\ref{J has only 3 roots}.
\qed

With a similar proof replacing  Lemma~\ref{lmn-irr-rep} by
 Lemma~\ref{lmp-irr-rep},  we have
\begin{lem}\label{(l,m,p) is different from (l',m')}The $A$-polynomial of any $k(l,m,0,p)$, $p\ne 0,1$, $l$ not divisible by $2p-1$, does not divide the $A$-polynomial of any
 $k(l',m',0,0)$.
\end{lem}

\begin{lem}\label{non-small Seifert surgery}
Let $K$ be a hyperbolic knot in $S^3$. For a given slope $p/q$, $p$ odd, if every solution $(x_0,y_0)$
of the system
of equations $$\left\{\begin{array}{ll}A_\sk(x,y)&=0,\\
x^py^q-1&=0,\end{array}\right.
$$ has  $x_0\in\{1,-1,0\}$, then
$M_\sk(p/q)$ cannot be  a Seifert fibred space whose base orbifold
is a $2$-sphere with exactly three cone points.
\end{lem}

\pf The proof is similar to that of Lemma~\ref{nonhyp surgery}.
Suppose otherwise that $M_\sk(p/q)$ is a a Seifert fibred space whose base orbifold
is a $2$-sphere with exactly three cone points.
As $p$ is odd, the base orbifold is either spherical or
hyperbolic.
Hence $\pi_1(M_\sk(p/q))$ has  an irreducible $PSL_2(\c)$ representation
 $\bar \r_0$  which factors
 through the orbifold fundamental group of the base orbifold
 such that the image of $\bar\r_0$ contains no parabolic elements.
 As $p$ is odd, the  $PSL_2(\c)$ representation
 $\bar \r_0$ lifts to a $SL_2(\c)$ representation $\r_0$.
 It is well known that
  the character $\chi_{\r_0}$ of $\r_0$ is an isolated point
in $X(M_\sk(p/q))$.  Considered as a point in $X(M_\sk)$,
  $\chi_{\r_0}$ is contained in a positive dimensional component
   $X_0$ of  $X(M_\sk)$.
Arguing exactly as in Claim~\ref{claim:non-constant}
 we have $f_{\m^p\l^q}$ is
non-constant on  $X_0$.
Hence  $X_0$ is contained $X^*(M_\sk)$ and contributes
a factor $f_0(x,y)$ to $A_\sk(x,y)$.
Exactly as in the proof of
Lemma~\ref{nonhyp surgery},
the point $\chi_{\r_0}$ provides a solution $(x_0, y_0)$
to  the system
$$\left\{\begin{array}{ll}A_\sk(x,y)&=0,\\
x^py^q-1&=0.\end{array}\right.$$ such that $x_0\notin \{1,-1,0\}$,
giving a contradiction with the assumption of the lemma.
\qed


\section{Distinguishing $k(l_*,-1,0,0)$ from $k(l,m,0,p)$}\label{distinguish J from k(l,m,0,p)}

The goal of this section is to prove the following proposition:

\begin{prop}\label{prop:l'-100}
Suppose that two knots $k(l_*,-1,0,0)$ and $k(l,m,0,p)$ have the same  genus
$g$ and the same half-integral toroidal slope $r$, where $l_*>1$ and $(1-2p)|l$, then $k(l_*,-1,0,0)=k(l,m,0,p)$.
\end{prop}

Let $$s=r+\frac12,\quad d=-s-2g.$$ For $k(l,m,0,p)$, by  (\ref{eq:rSlope}),
\begin{equation}\label{eq:s0p}
-s=-p(2ml-l-1)^2+(2ml-l)(ml-1).
\end{equation}
When $p\le0$, using (\ref{eq:Genus0p}) , we get
\begin{equation}\label{eq:d}
d=\left\{\begin{array}{ll}
-p(2ml-l-1)+3ml-l-2\alpha, &\text{if } lm>0,\\
-p(-2ml+l+1)-3ml+l, &\text{if } lm<0,
\end{array}\right.
\end{equation}
where
$$
\alpha=\left\{\begin{array}{ll}
1, &\text{if }l>0,m>0,\\
2, &\text{if }l<0,m<0.
\end{array}\right.
$$

Consider
 the family of knots $k(l_*,-1,0,0)$, where $l_*>1$.
In this case, $d,s$ are given by
\begin{equation}\label{eq:ds-100}
\left\{
\begin{array}{rcl}
d&=&4l_*\\
s&=&-3l_*(l_*+1),
\end{array}
\right.
\end{equation}
so
\begin{equation}\label{eq:sd}
-\frac43s+1=(\frac d2+1)^2.
\end{equation}

\begin{lem}\label{lem:p<0}
Suppose that $p<0$ and $(1-2p)|l$, then the knot $k(l,m,0,p)$ has different $(g,r)$ from the knot $k(l_*,-1,0,0)$, where $l_*>1$.
\end{lem}
\begin{proof}
Otherwise, assume $k(l,m,0,p)$ has the same $(g,r)$ as $k(l_*,-1,0,0)$.
Using (\ref{eq:d}), we see that
$$
d+2>-p|2ml-l-1|+|2ml-l-1|.
$$
Using (\ref{eq:s0p}) and (\ref{eq:lmCondition}), we have
$$
-s+\frac34<-p(2ml-l-1)^2+(2ml-l)(ml-1)+1<(1-p)(2ml-l-1)^2.
$$
Using (\ref{eq:sd}) and the previous two inequalities, we get
\begin{eqnarray*}
\frac{(1-p)^2}4(2ml-l-1)^2&<&(\frac d2+1)^2\\
&=&-\frac43s+1\\
&<&\frac43(1-p)(2ml-l-1)^2,
\end{eqnarray*}
so $1-p<\frac{16}3$, hence $-p\le4$.

If $p=-1$ or $-4$, $3|l$. By (\ref{eq:s0p}), $3\nmid -s$, which contradicts (\ref{eq:ds-100}).

If $p=-3$, then $7|l$. By (\ref{eq:s0p}), $-s\equiv3\pmod7$. It follows from (\ref{eq:sd}) that $5$ is a quadratic residue modulo $7$, which is not true.

If $p=-2$, then $5|l$ and $-s\equiv2\pmod5$.  It follows from (\ref{eq:sd}) that $2$ is a quadratic residue modulo $5$, which is not true.
\end{proof}

For the family of knots $k(l,m,0,0)$, the following proposition expresses $l,m$ in terms of $s$ and $d$.

\begin{prop}\label{prop:lmFormula}
For the family of knots $k(l,m,0,0)$, the pair $(l,m)$ is determined by $s$ and $d$ by the following formulas:\newline
If $lm>0$, then
\begin{eqnarray*}
l&=&\frac{d+2\alpha+3-3\sqrt{(d+2\alpha-1)^2+4s}}2,\\
m&=&\frac13\left(\frac{2d+4\alpha}{d+2\alpha+3-3\sqrt{(d+2\alpha-1)^2+4s}}+1\right).
\end{eqnarray*}
If $lm<0$, then
\begin{eqnarray*}
l&=&\frac{-d+3+3\sqrt{(d+1)^2+4s}}2,\\
m&=&\frac13\left(1-\frac{2d}{-d+3+3\sqrt{(d+1)^2+4s}}\right).
\end{eqnarray*}
\end{prop}
\begin{proof}
Using (\ref{eq:d}), we can express $ml$ as a linear function of $d$ and $l$. Substitute such expression of $ml$ into (\ref{eq:s0p}), we get
\begin{equation}\label{eq:lQuad}
-s=\left\{
\begin{array}{ll}
\frac19(-l+2d+4\alpha)(l+d+2\alpha-3), &\text{if }lm>0,\\
&\\
\frac19(-l-2d)(l-d-3), &\text{if }lm<0.
\end{array}
\right.
\end{equation}

If $lm>0$,
by (\ref{eq:lQuad}), $l$ is a root of the quadratic polynomial
\begin{equation}\label{eq:Quad>0}
(x-2d-4\alpha)(x+d+2\alpha-3)-9s,
\end{equation}
whose two roots are
$$\frac{d+2\alpha+3\pm3\sqrt{(d+2\alpha-1)^2+4s}}2.$$
By (\ref{eq:d}), $d>0$ and $|2l|\le d+2\alpha$, so
$$l=\frac{d+2\alpha+3-3\sqrt{(d+2\alpha-1)^2+4s}}2.$$
Using (\ref{eq:d}) again, we can compute $m$ as in the statement.

If $lm<0$, $l$ is a root of the
quadratic polynomial
\begin{equation}\label{eq:Quad<0}
(x+2d)(x-d-3)-9s,
\end{equation}
whose two roots are
$$\frac{-d+3\pm3\sqrt{(d+1)^2+4s}}2.$$
By (\ref{eq:d}), $d>0$ and $|2l|\le d$, so
$$l=\frac{-d+3+3\sqrt{(d+1)^2+4s}}2.$$
Using (\ref{eq:d}) again, we can compute $m$ as in the statement.
\end{proof}

\begin{lem}\label{lem:lm>0}
Suppose that two knots $k(l,m,0,0)$ and $k(l_*,m_*,0,0)$ have the same $g$ and $r$. Suppose further that $$lm,l_*m_*>0,\qquad l>0>l_*.$$
Then the quadruple $(l,m,l_*,m_*)$ is either $(2,2,-3,-1)$ or
$(6,m,-2,-3m+1)$ for $m\ge2$.
\end{lem}
\begin{proof}
Using Proposition~\ref{prop:lmFormula} and (\ref{eq:lmCondition}), we get
\begin{equation}\label{eq:l_ge2}
l=\frac{d+5-3\sqrt{(d+1)^2+4s}}2\ge2
\end{equation}
and
\begin{equation}\label{eq:l'le-2}
l_*=\frac{d+7-3\sqrt{(d+3)^2+4s}}2\le-2.
\end{equation}
Hence
\begin{equation}\label{eq:4s}
\frac{(d+11)^2}9-(d+3)^2\le4s\le-\frac89(d+1)^2,
\end{equation}
which implies
$$\frac19(d-7)^2\le (d+1)^2+4s\le \frac19(d+1)^2.$$
From (\ref{eq:d}) and (\ref{eq:lmCondition}), we can conclude that
\begin{equation}\label{eq:d_ge8}
d\ge8.
\end{equation}
By (\ref{eq:l_ge2}), $(d+1)^2+4s$ is a perfect square which has the same parity as $d+1$, hence
\begin{equation}\label{eq:2c}
(d+1)^2+4s=\frac{(d+1-2c)^2}9,
\end{equation}
for some $c\in\{0,1,2,3,4\}$.
Then $$l=c+2,\quad m=\frac{d+c+4}{3(c+2)}.$$

Using (\ref{eq:4s}), we also get
$$\frac{(d+11)^2}9\le (d+3)^2+4s\le \frac{d^2+38d+73}9<\frac{(d+19)^2}9.$$
By (\ref{eq:l'le-2}), $(d+3)^2+4s$ is a perfect square with the same parity as $d+3$, so
\begin{equation}\label{eq:2c'}
(d+3)^2+4s=\frac{(d+11+2c_*)^2}9
\end{equation}
for some $c_*\in\{0,1,2,3\}$.
Then
$$l_*=-c_*-2,\quad m_*=-\frac{d-c_*+2}{3(c_*+2)}.$$
Compare (\ref{eq:2c}) and (\ref{eq:2c'}), we get
\begin{equation}\label{eq:dcc'}
(4-c_*-c)d=(5+c_*+c)(6+c_*-c)-18.
\end{equation}
Moreover, since both $\frac{d+1-2c}3$ and $\frac{d+11+2c_*}3$ are integers, $\frac{10+2c_*+2c}3$ is an integer, so
$$
c+c_*\in\{1,4,7\}.
$$

If $c+c_*=7$, then $c=4,c_*=3$. By (\ref{eq:dcc'}), $d=-14$, a contradiction to (\ref{eq:d_ge8}).

If $c+c_*=4$, using (\ref{eq:dcc'}) we get $c=4,c_*=0$. Hence $l=6, m=\frac{d+8}{18}$, $l_*=-2,m_*=-\frac{d+2}6=-3m+1$.

If $c+c_*=1$, using (\ref{eq:dcc'}) we get $d=6+2(c_*-c)$. Using (\ref{eq:d_ge8}), the only possible case is
$c=0,c_*=1$, and $(l,m,l_*,m_*)=(2,2,-3,-1)$.
\end{proof}

\begin{proof}[Proof of Proposition~\ref{prop:l'-100}]
When $p<0$, this result follows from Lemma~\ref{lem:p<0}.

When $p>0$, from the formula (\ref{eq:s0p}) it is easy to see $s>0$ for $k(l,m,0,p)$, but $s=-3l_*(l_*+1)<0$ for $k(l_*,-1,0,0)$.

Now we consider the case $p=0$.
By Proposition~\ref{prop:lmFormula}, for any given $(g,r)$, there are at most three knots $k(l_i,m_i,0,0)$, $i=1,2,3$, having this $(g,r)$. There is at most one pair $(l_i,m_i)$ in each of the three cases:
\begin{itemize}
\item $l>0,m>0,$
\item $l<0,m<0,$
\item $lm<0.$
\end{itemize}
The pair $(l_*,-1)$ is in the third case. By Proposition~\ref{prop:Dupl} (c), $k(-l_*-1,-1,0,0)=k(l_*,-1,0,0)$, and $(-l_*-1,-1)$ is in the second case. Suppose that there is also a pair $(l,m)$ in the first case with the same $(g,r)$, then
Lemma~\ref{lem:lm>0} implies that $(l,m)=(2,2)$ and $-l_*-1=-3$. By Proposition~\ref{prop:Dupl} (d), the three knots $k(2,2,0,0),k(-3,-1,0,0),k(2,-1,0,0)$ are the same.
\end{proof}


\section{Proof of Theorem~\ref{main2}}\label{proof2}

Recall that the unique half-integral toroidal slope of
$k(l_*,-1,0,0)$ is
$r=(2s-1)/2$, where $s=-3l_*(l_*+1)$.
Hence the family of  knots $\{k(l_*,-1,0,0); l_*>1\}$
are mutually distinct.

Suppose that $K\subset S^3$ is a knot which has the same $A$-polynomial and
the same knot Floer homology as a  given $J_*=k(l_*,-1,0,0)$, $l_*>1$.
Our goal is to show that $K=J_*$.

As $J_*$ is hyperbolic, $K$ cannot be a torus knot by Theorem~\ref{main1}.

Suppose $K$ is hyperbolic.
Then by Lemma~\ref{same r slope},
 $K$ has the same half-integral toroidal slope
 $r$  as $J_*$
and $K$ is one of $k(l,m,n,p)$.
Applying Lemma~\ref{(l,m,n) is different from (l',m')}, we have
$n=0$ or $1$.
Since the half-integral toroidal slope of
 $k(l,m,1,0)$ is positive while the half-integral toroidal slope
 of $J_*=k(l_*,-1,0,0)$ is negative, we see that $n$ must be zero.
Similarly applying Lemma~\ref{(l,m,p) is different from (l',m')}, we have $p$ is non-positive, and
$2p-1$ divides $l$.
That is, we have $K=k(l,m,0,p)$ with $(2p-1)|l$.
Now by Proposition~\ref{prop:l'-100}, we have
$K=J_*$.

It remains to show that
$K$ cannot be a satellite knot.
Suppose otherwise that $K$ is a satellite knot.
We are going to derive a contradiction from this assumption.

\begin{lem}\label{A_J has no seminorm factor}
The $A$-polynomial $A_{\sj_*}(x,y)$ of $J_*=k(l_*,-1,0,0)$ does not contain
any factor of the form  $x^{j}y+\d$
or $y+\d x^{-j}$ for  $j\in\z$ and $\d\in\{-1,1\}$.
\end{lem}

\pf Suppose otherwise that $A_{\sj_*}(x,y)$    contains
a  factor of the form   $x^{j}y+\d$
or $y+\d x^{-j}$. As this factor is irreducible and balanced,
it is contributed by  a curve component $X_0$ in $X^*(M_{\sj_*})$.
Moreover $X_0$ is a semi-norm curve with
 $\m^j\l$ as the unique associated boundary slope.

 We claim that
   $\m^j\l$ is  either $\m^s\l$ or $\m^{s-1}\l$.
From Theorem~\ref{properties of a semi-norm curve} (3) and Lemma~\ref{meridian not bdy slope}, we see that the meridian slope
$\m$ has the minimal semi-norm $s_0$.
To prove the claim, we just need to show, by Theorem~\ref{properties of a semi-norm curve} (4),  that for the half-integral toroidal slope $r$ of $J_*$, we have
$\|r\|_0=\|\m\|_0$, which is equivalent to show
that $$Z_v(\tilde f_r)\leq Z_v(\tilde f_\m)$$ for every
$v\in \tilde X_0$.
As $M_{\sj_*}(r)$ has no non-cyclic representation by
Lemma~\ref{prop:AbelRep},
at every regular point $v\in\tilde X_0$ we have
$Z_v(\tilde f_r)\leq Z_v(\tilde f_\m)$.
If at an ideal point $v$ of $\tilde X_0$,
$Z_v(\tilde f_r)> Z_v(\tilde f_\m)$,
then $\tilde f_\a(v)$ is finite for every class
$\a$ in $H_1(\p M_{\sj_*}; \z)$ (because both $\tilde f_{r}$
and $\tilde f_{\m^j\l}$ are finite at $v$).
We can now derive a contradiction with  \cite[Proposition~4.12]{BZ4}
just as in the proof of Lemma~\ref{r is a vertex}, and the claim is thus proved.

On the other hand  it is shown in \cite[Proposition~5.4]{EM2} that each of
$M_{\sj_*}(s)$ and $M_{\sj_*}(s-1)$ is a small Seifert
fibred space. As $J_*$ is a small knot, each of  $\m^s\l$ and $\m^{s-1}\l$
 cannot be a boundary slope
by \cite[Theorem~2.0.3]{CGLS} and thus cannot be the slope $\m^j\l$.
We arrive at a contradiction.
\qed

Since $J_*$ is a fibered knot by \cite{EM2}, $K$ is also fibered.
Hence if $(C,P)$ is any pair of companion knot and pattern knot associated to $K$, then
each of $C$ and $P$ is fibred and the winding number $w$ of $P$ with respect to $C$
is larger than zero.

\begin{lem}\label{hyp companion}
The  satellite knot  $K$ has a companion knot $C$
which is hyperbolic.
\end{lem}

\pf As is true for any satellite knot,
$K$  has a  companion knot $C$ which is either a torus
knot or a hyperbolic knot.
So we just need to rule out the possibility that
$C$ be a torus knot.
If  $C$ is a $(p,q)$-torus knot, then
by formula (\ref{Apoly of T(p,q)}), $A_\sc(x,y)$ contains a factor of the form
$yx^k+\delta$
or $y+\delta x^{-k}$ for some integer $k$ and
$\delta\in\{1,-1\}$.
As the winding number $w$ of the pattern knot $P$ with respect to $C$ is non-zero,
 by Proposition~\ref{non-zero winding},
$A_\sk(x,y)$ contains a factor of the form
$yx^{w^2k}-(-\delta)^w$
or $y-(-\delta)^w x^{-w^2k}$.
But this contradicts  Lemma~\ref{A_J has no seminorm factor}.
\qed

We now fix a hyperbolic companion knot $C$ for $K$ which exists
by Lemma~\ref{hyp companion}
and let $P$ be the corresponding pattern knot.

\begin{lem}\label{surgery on C}
For any hyperbolic knot $C$ in $S^3$, any surgery with a slope $j/k$, $\gcd(j,k)=1$, satisfying
$k>2$ and $j$ odd, will produce either a hyperbolic manifold or a  Seifert
fibred space whose base orbifold is $S^2$ with exactly three cone points.
\end{lem}

\pf
As $k>2$, $M_\sc(j/k)$ is irreducible \cite{GL} and atoroidal
\cite{GL2}.
Thus  $M_\sc(j/k)$ is either a hyperbolic manifold or an atoroidal Seifert fibred
space.
In latter case, the Seifert fibered space  has non-cyclic fundamental group
\cite{CGLS}.
As $p$ is odd, the base orbifold  of the Seifert fibered space cannot be non-orientable.
Thus the base orbifold is a $2$-sphere with exactly three cone points.
\qed

\begin{lem}\label{w divides 2s-1}
The integer $d=2s-1$ is divisible by  $w^2$.
\end{lem}

\pf Suppose otherwise.  Let $d_1/q_1$ be the rational  number $\frac{d}{2w^2}$ in its reduced form, i.e.  $d_1=d/\gcd(d, w^2)$ and $q_1=2w^2/\gcd(d, w^2)$.
Then $q_1>2$ and $d_1$ is odd.
So by Lemma~\ref{surgery on C}, the  surgery on $C$ with the slope $d_1/q_1$ will yield
either a hyperbolic manifold or a Seifert fibred space whose
base orbifold is a $2$-sphere with exactly three cone points.
Applying either Lemma~\ref{nonhyp surgery} or
Lemma~\ref{non-small Seifert surgery}, we see that
the $A$-polynomial $A_\sc(\bar x,\bar y)$ of $C$ has a zero point
$(\bar x_0,\bar y_0)$ such that $\bar x_0^{d_1}\bar y_0^{q_1}=1$
and $\bar x_0\notin \{0,1,-1\}$.
Now from  Proposition~\ref{non-zero winding}
and its proof, we see that
$A_\sc(\bar x,\bar y)$
can be extended to a factor $f(x,y)$ of $A_\sk(x,y)$
and the variables of $A_\sc(\bar x,\bar y)$ and $f(x,y)$ satisfying
the relations: $\bar x=x^w$, $\bar y^w=y$.
In particular for some $(x_0, y_0)$ we have
$\bar x_0=x_0^w$, $\bar y_0^w=y_0$, and $(x_0,y_0)$ is a zero point of $f(x,y)$.
Obviously $x_0\notin \{0,1,-1\}$.
From  $(\bar x_0^{d_1}\bar y_0^{q_1})^w=1$, we have
$x_0^{w^2d_1}y_0^{q_1}=1$, i.e. $x_0^{d}y_0^{2}=1$.

 As $f(x,y)$ is a factor in $A_\sk(x,y)=A_{\sj_*}(x,y)$,
we see that the system
$$\left\{\begin{array}
{ll}
A_{\sj_*}(x,y)&=0\\
x^{d}y^2-1&=0\end{array}\right.$$
has a solution $(x_0, y_0)$ with $x_0\notin\{0, 1, -1\}$.
We get a contradiction with Lemma~\ref{J has only 3 roots}.
\qed

Note that $s-1$ is a cyclic  surgery slope of $J_*$ (provided by \cite[Proposition~5.4]{EM2}).

\begin{lem}\label{s-1 has only 2 roots}
If $(x_0,y_0)$ is a solution of the system
$$\left\{\begin{array}
{ll}
A_{\sj_*}(x,y)&=0\\
x^{s-1}y-1&=0\end{array}\right.$$
then $x_0$ is either $1$ or $-1$.
\end{lem}

\pf If   $x_0\ne 1$ or $-1$,  it implies  that
there is a  curve component $X_0$  in $X^*(M)$
such that $\tilde X_0$
has a point at which $\tilde f_{\m^{s-1}\l}=0$
but $\tilde f_\m\ne 0$.
This is impossible as $\tilde f_{\m^{s-1}\l}$ has the
minimal zero degree at every point of $\tilde X_0$
(because $J_*$ is a small knot, $s-1$ is a cyclic surgery slope but is not a boundary slope).
\qed

\begin{lem}\label{w divides s-1}
The integer $s-1$ is divisible by  $w^2$.
\end{lem}

\pf The proof is similar to that of Lemma~\ref{w divides 2s-1}, only replacing
Lemma~\ref{J has only 3 roots} by Lemma~\ref{s-1 has only 2 roots}.
First note that $s-1$ is an odd number (as $s=-3l_*(l_*+1)$ is even).
So if $s-1$ is not divisible by $w^2$, then
the reduced form $d_1/q_1$ of the rational number
$(s-1)/w^2$ has denominator $q_1>2$.
Now arguing as in the proof of Lemma~\ref{w divides 2s-1}
 starting from the $d_1/q_1$-surgery on $C$, we
 see that
the system
$$\left\{\begin{array}
{ll}
A_{\sj_*}(x,y)&=0\\
x^{s-1}y-1&=0\end{array}\right.$$
has a solution $(x_0, y_0)$ with $x_0\ne  1, -1$.
This gives a contradiction with Lemma~\ref{s-1 has only 2 roots}.\qed

\begin{cor}\label{w is one}The winding number $w=1$.
\end{cor}

\pf This follows immediately
from Lemmas~\ref{w divides 2s-1} and \ref{w divides s-1}\qed

\begin{lem}\label{C and J have the same r}
The companion knot $C$ has the same half-integral toroidal slope as
$J_*$ and $C$ is one of $k(l,m,0,p)$ with $p$ non-positive and with $(2p-1)|l$.
\end{lem}

\pf It follows from  Corollary~\ref{w is one}
and Proposition~\ref{non-zero winding} that
$A_\sc(x,y)$ is a factor of $A_\sk(x,y)=A_{\sj_*}(x,y)$.
So  Proposition~\ref{same r slope} says that $r$ is
 also a toroidal slope of $C$ and $C$ is one of $k(l,m,n,p)$.
 Now arguing as in the case when $K$ is hyperbolic, we see
 that $C$ is one
of $k(l,m,0,p)$  with $p$ non-positive and with $(2p-1)|l$.
\qed

\begin{lem}\label{s-1 is cyclic for C}
For $C=k(l,m,0,p)$ given by Lemma~\ref{C and J have the same r},  we have
$m=-1$ unless $C=k(2,2,0,0)$ or $C=k(-2,m,0,0)$.
\end{lem}

\pf
If $m\ne -1$, then by
\cite[Theorem~2.1 (d)]{EM1} the $(s-1)$-surgery on $C=k(l,m,0,p)$
is a Seifert fibred manifold whose base orbifold is
a $2$-sphere with exactly three cone points, except
when $C$ is one of  $k(-2,m,0,p)$,
$k(2,2,0,0)$,  $k(2,3,0,1)$, $k(3,2,0,1)$,
$k(2,2,0,2)$. As we knew that $p$ is non-positive and
$2p-1$ divides $l$,  these exceptional cases
can be excluded except $k(2,2,0,0)$ or $k(-2,m,0,0)$.
So we just need to deal with the case
 when the $(s-1)$-surgery on $C$
is a Seifert fibred manifold whose base orbifold is
a $2$-sphere with exactly three cone points.
Note that $s-1$ is odd.
Hence by Lemma~\ref{non-small Seifert surgery},
the system
$$\left\{\begin{array}
{ll}
A_\sc(x,y)&=0\\
x^{s-1}y-1&=0\end{array}\right.$$
has a solution $(x_0, y_0)$ with $x_0\ne  1, -1$.
As $A_\sc(x,y)$ is a factor of $A_{\sj_*}(x,y)$,
the point $(x_0, y_0)$ is also a solution of the system
$$\left\{\begin{array}
{ll}
A_{\sj_*}(x,y)&=0\\
x^{s-1}y-1&=0\end{array}\right.$$
which yields a contradiction with Lemma~\ref{s-1 has only 2 roots}.
\qed

\begin{lem}\label{less than 4g}$$|s-2|\leq 4g(C).$$
\end{lem}

\pf Recall that $C$ is one of $k(l,-1,0,p)$, $p$ non-positive,  $(2,2,0,0)$, $k(-2,m,0,0)$,
and has the same $r$ slope as $J_*$ and thus has the same $s$ slope  as $J_*$.
From (\ref{eq:rSlope}) and (\ref{eq:Genus0p}) we have that
for $k(l,-1,0,p)$, $p$ non-positive, its $s$ slope and genus $g$ are
given by $$s=-3l(l+1)+p(-3l-1)^2,\; g=\left\{\begin{array}{ll}\frac{-p(3l+1)3l}{2}+l^2+\frac{l(l-1)}{2}, &\mbox{if $l>0$},
\\\frac{-p(-3l-1)(-3l-2)}{2}+l^2+\frac{l(l+5)}{2}+l+2, &\mbox{if $l<0$};\end{array}\right.$$
for $k(2,2,0,0)$ its $s$ and $g$ values are
$$s=-18,\; g=5;$$ and for $k(-2,m,0,0)$,
$$s=-2(2m-1)(2m+1),\; g=\left\{\begin{array}{ll}4m^2-3m, &\mbox{if $m>0$},
\\4m^2+3m, &\mbox{if $m<0$}.\end{array}\right.$$
In each case, one can check directly that
$|s-2|\leq 4g$ holds, being aware of some forbidden values on $l,m,p$ given by
(\ref{eq:lmCondition}).
\qed

As noted in the proof of Theorem~\ref{main1},
when the winding number $w=1$, the pattern knot $P$ is a nontrivial knot.
We also have, by Lemma~\ref{P-factor}, that $A_\sp(x,y)$ is a factor of
$A_\sk(x,y)=A_{\sj_*}(x,y)$.
Combining this fact with  (\ref{Apoly of T(p,q)}) and Lemma~\ref{A_J has no seminorm factor}
we know that $P$ cannot be a torus knot.
So $P$ is either a hyperbolic knot or a satellite knot.

If $P$ is a hyperbolic knot, then arguing as in the proofs of Lemmas \ref{C and J have the same r},
\ref{s-1 is cyclic for C} and \ref{less than 4g},
we have that $P$ has the same half-integral toroidal slope as
$J_*$,  that $P$ is one of $k(l',-1,0,p')$ with $p'$ non-positive or $k(2,2,0,0)$ or $k(-2,m',0,0)$,
and that $|s-2|\leq 4g(P)$.

Now from
$$\D_{J_*}(t)=\D_\sk(t)=\D_\sc(t)\D_{\sp}(t)$$
we have  that the genus of the given satellite knot
$K$ (which is equal to that of $J_*$)
 is equal to the sum of the genus of $C$ and the genus of $P$.
So the genus of one of $C$ and $P$, say $C$ (same argument for $P$), is less than or equal to
the half of the genus of $J_*$, i.e.
$$g(C)\leq \frac12g(J_*).$$
So \begin{equation}\label{s inequality} |s-2|\leq 2g(J_*).\end{equation}
But $s=-3l_*(l_*+1)$ and $g(J_*)=l_*^2+\frac{l_*(l_*-1)}{2}$, which do not fit in
the inequality (\ref{s inequality}).
This  contradiction shows that  $P$ cannot be hyperbolic.

So $P$ is a satellite knot.
Let $(C_1, P_1)$ be a pair companion knot and pattern knot for $P$.
Once again as $P$ is fibred, each of $C_1$ and $P_1$ is fibred, and
the winding number $w_1$ of $P_1$ with respect to $C_1$
is larger than zero.
Making use of the fact that $A_\sp(x,y)|A_{\sj_*}(x,y)$, one can
show, similarly  as for the pair $(C,P)$,  that $C_1$
can be assumed to be hyperbolic, that $w_1=1$, that $C_1$ has the same $r$ and $s$ values as $J_*$,
 that $C_1=k(l'', -1, 0, p'')$ for some non-positive $p''$
or $k(2,2,0,0)$ or $k(-2,m'',0,0)$, and that $|s-2|\leq 4g(C_1)$.
Now from the equality
$$\D_{J_*}(t)=\D_\sk(t)=\D_\sc(t)\D_{\sc_1}(t)\D_{\sp_1}(t)$$
we see that one of
$g(C)$ and $g(C_1)$ is less than or equal to $\frac12g(J_*)$.
This leads to a contradiction just as in the preceding
paragraph. So $P$ cannot be a satellite knot, and this final
contradiction completes the proof of Theorem~\ref{main2}.

\end{document}